\documentclass[a4paper]{article}

\usepackage[utf8]{inputenc}

\pdfoutput=1
\usepackage[margin=1in]{geometry}

\usepackage[parfill]{parskip}
\usepackage[british]{babel}
\usepackage[utf8]{inputenc}
\usepackage[T1]{fontenc}
\usepackage{lmodern}
\usepackage{sectsty}
\usepackage{booktabs}
\usepackage{amsmath,amssymb,amsthm,amsfonts}
\usepackage{mathtools}
\usepackage[pdftex]{graphicx}
\usepackage{xspace}
\usepackage{float}
\restylefloat{table}
\usepackage{bibentry}
\usepackage{authblk}
\usepackage{subcaption}
\usepackage{qcircuit}
\usepackage{bbm}
\usepackage[bookmarksnumbered,bookmarksopen,unicode,colorlinks,allcolors=blue]{hyperref}
\usepackage{algorithm}
\usepackage[noend]{algpseudocode}
\algrenewcommand\algorithmicrequire{\textbf{Input:}}
\algrenewcommand\algorithmicensure{\textbf{Output:}}
\usepackage{cleveref}
\usepackage{url}

\newcommand{\N}{\mathbb{N}}

\newtheorem{theorem}{Theorem}[section]
\newtheorem{corollary}[theorem]{Corollary}
\newtheorem{lemma}[theorem]{Lemma}
\newtheorem{proposition}[theorem]{Proposition}
\newtheorem{conjecture}[theorem]{Conjecture}

\theoremstyle{definition}
\newtheorem{definition}[theorem]{Definition}

\newtheorem{example}[theorem]{Example}
\newtheorem{problem}[theorem]{Problem}

\DeclareMathOperator{\conv}{\mathrm{conv}}

\newcommand{\MT}{\text{\rm MT}}
\newcommand{\ML}{\text{\rm MS}}
\newcommand{\define}{\mathrel{{\mathop:}{=}}}
\allsectionsfont{\normalsize}

\makeatletter
\newcommand{\setword}[2]{%
  \phantomsection
  #1\def\@currentlabel{\unexpanded{#1}}\label{#2}%
}
\makeatother

\title{Optimized Qubit Routing for Commuting Gates via\\Integer Programming}

\author[1,*]{Moritz Stargalla}
\author[2]{Friedrich Wagner}
\affil[1]{Applied Discrete Mathematics Lab, University of Technology Nuremberg}
\affil[2]{Fraunhofer Institute for Integrated Circuits, Nuremberg}
\affil[*]{\texttt{\small moritz.stargalla@utn.de}}
\date{}

\begin{document}

\maketitle

\begin{abstract}
\noindent
Quantum computers promise to outperform their classical counterparts at certain tasks.
However, existing quantum devices are error-prone and restricted in size.
Thus, effective compilation methods are crucial to exploit limited quantum resources.
In this work, we address the problem of qubit routing for commuting gates, which arises, for example, during the compilation of the well-known Quantum Approximate Optimization Algorithm.
We propose a two-step decomposition approach based on integer programming, which is guaranteed to return an optimal solution.
To justify the use of integer programming, we prove NP-hardness of the underlying optimization problem.
Furthermore, we derive asymptotic upper and lower bounds on the quality of a solution.
We develop several integer programming models and derive linear descriptions of related polytopes, which generalize to applications beyond this work.
Finally, we conduct a computational study showing that our approach outperforms existing heuristics in terms of quality and exact methods in terms of runtime.
\end{abstract}
\section{Introduction}
Quantum computing bears the potential to impact various fields such as chemistry~\cite{Bauer_2020}, material science~\cite{barkoutsos2021quantum}, machine learning~\cite{cerezo2022} and optimization~\cite{Abbas_2024}.
However, current quantum hardware is limited in size and prone to high error rates~\cite{preskill2018quantum}.
Thus, it is crucial to efficiently exploit limited quantum resources and to reduce error probabilities.
\emph{Qubit routing} is an $\mathcal{NP}$-hard combinatorial optimization problem which arises when compiling quantum algorithms to hardware-specific instructions~\cite{nannicini2021optimalqubitassignmentrouting,Wagner_2023}.
The quality of the qubit routing solution heavily influences the success probability during execution~\cite{Murali2019,wagner2025optimized}.

Quantum algorithms are generally formulated in the circuit model of quantum computing, see~\cite{Nielsen_Chuang_2010} for an introduction.
An example quantum circuit is depicted in \Cref{fig:intro_ex1}.
Such a circuit comprises \emph{wires} representing \emph{logical qubits} and a temporal sequence of \emph{gates} applied to subsets of logical qubits.
In this work, we consider circuits containing exclusively gates that act on pairs of qubits, called two-qubit gates.
The presence of single-qubit gates does not alter the resulting qubit routing problem.
Moreover, any multi-qubit gate can be decomposed into single- and two-qubit gates~\cite{dawson2005solovay}.

A quantum circuit may apply two-qubit gates on any pair of logical qubits.
However, actual quantum processors can only apply two-qubit gates on certain pairs of qubits.
Thus, quantum processors are often represented by a \emph{hardware graph}, see \Cref{fig:intro_ex1} for an example.
The nodes of the hardware graph correspond to \emph{physical qubits} while its edges define the pairs of physical qubits on which two-qubit gates can be applied.

\textbf{General qubit routing.}
To execute a circuit on a quantum computer with a given hardware graph, one must choose an initial allocation of logical to physical qubits ensuring that for each two-qubit gate, the logical qubits are located at neighboring physical qubits in the hardware graph.
However, such an allocation does not exist in general, as is the case for the example in \Cref{fig:intro_ex1}.
In such a case, additional \emph{swap gates} have to be inserted into the circuit.
A swap gate is a two-qubit gate which effectively swaps the allocations of the involved logical qubits.
Thus, swap gates allow to modify the initial allocation of logical to physical qubits.
The outcome is a routed circuit that meets the connectivity restrictions and is equivalent to the original circuit up to a permutation of the qubits, as shown in \Cref{fig:intro_ex1}.
This permutation results from the initial allocation and the swap gates.

\begin{figure}
    \centering
    \includegraphics[page=1, width=\textwidth]{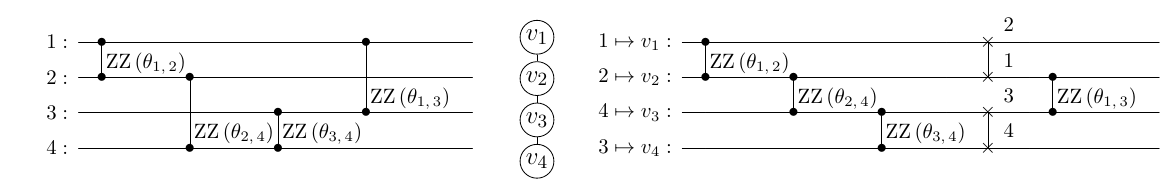}
	\caption{
	Example of a quantum circuit with commuting gates (left),
        a hardware graph (center),
        and a routed quantum circuit with inserted swap gates (right).
        For each gate in the routed circuit, the involved pair of logical qubits gates is allocated at neighboring physical qubits.
        The order of gates in the routed circuit is the same as in the original circuit.
	}
	\label{fig:intro_ex1}
\end{figure}

Due to high gate error rates in existing quantum devices, it is crucial to minimize the number of swap gates in the routed circuit.
Moreover, the error probability increases with the time taken to execute the routed circuit.
This translates into minimizing the \emph{depth} of the circuit.
The depth is defined as the minimum number of \emph{layers} the gates can be partitioned into.
A layer is a set of consecutive gates that act on disjoint qubits and can thus be executed in parallel.
For example, the depth of the routed circuit in \Cref{fig:intro_ex1} amounts to $5$.

\textbf{Qubit routing with commuting gates.}
Generally, the order of the two-qubit gates in a circuit matters.
However, important applications such as the Quantum Approximate Optimization Algorithm (QAOA)~\cite{farhi2014quantum} and Hamiltonian simulation~\cite{CarreraVazquez2023wellconditioned} result in circuits with blocks of \emph{commuting} gates.
Thus, inside one block the order of gates can be changed arbitrarily.
In such a case, the order of the gates provides an additional degree of freedom that can be optimized, potentially reducing the number of swap gates and the circuit depth as illustrated in the example of \Cref{fig:intro_circuit_swap_comm}.

\begin{figure}
    \centering
	\includegraphics[page=2, width=0.375\textwidth]{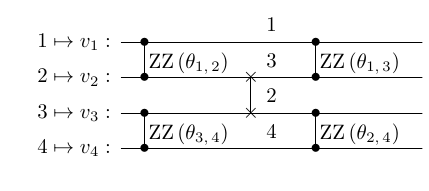}
	\caption{
		A routed quantum circuit that is equivalent to the circuits in \Cref{fig:intro_ex1}
        but exploits commutativity.
        Both depth and swap count are reduced compared to the routed circuit in \Cref{fig:intro_ex1}.
	}
	\label{fig:intro_circuit_swap_comm}
\end{figure}

\subsection{Our Contribution}
We propose an algorithm for the qubit routing problem with commuting gates that is guaranteed to return a circuit with a minimum number of swaps.
This is in contrast to existing heuristic methods, which do not provide any performance guarantees.
Moreover, our algorithm also heuristically minimizes depth among all swap-optimal solutions.
We achieve this by developing a two-stage optimization approach that lexicographically optimizes for swap count and depth.
First, we compute an initial mapping of circuit to physical qubits and a set of swap layers ensuring that each gate is executable.
In the second stage, we schedule all gates of the original circuit into either a swap layer or a new layer, aiming to minimize the depth of the resulting circuit.
We formally define the first problem of our two-step approach as a combinatorial optimization problem, which we call \emph{Token Meeting Problem} (TMP).
We analyze the TMP and derive asymptotic lower and upper bounds on the cost of a solution.
After proving $\mathcal{NP}$-hardness, we propose several integer programming models for solving the TMP.
Furthermore, we provide linear descriptions for two associated polytopes, which may have applications beyond this work.
Additionally, we develop an integer programming model also for the second stage of our approach.
Finally, we conduct computational experiments which show that our methods outperform existing heuristics in terms of quality and existing exact methods in terms of solution time.

\subsection{Related Work}
For an introduction to quantum computing, we refer to \cite{Nielsen_Chuang_2010}.
Ge et al.~\cite{ge2024quantum} provide a comprehensive review on the synthesis and compilation of quantum circuits.

Numerous heuristics have been proposed for the general qubit routing problem \cite{zulehner2018efficient, cowtan2019qubit,li2019tackling,sivarajah2020t,tan2025compilation, sun2025haqa}, see \cite{tan2020optimality} for a detailed survey.
Heuristics typically scale well with respect to the circuit size and are thus the method of choice for scalable and highly-automated compilation pipelines.
Arguably, one of the most popular heuristics is the SABRE method by Li et al.~\cite{li2019tackling}, which leverages the reversibility of quantum circuits to perform a bidirectional search for a good initial mapping, thereby reducing the necessary swap gate count.

Existing exact methods for general qubits routing are based on satisfiability (SAT) \cite{wille2019mapping, molavi2022qubit, yang2024quantum, shaik2024optimal, shaik2024optimal2}, satisfiability modulo theory (SMT) \cite{tan2020optimal, lin2023scalable} or integer programming \cite{nannicini2021optimalqubitassignmentrouting,Wagner_2023,wagner2025optimized}.
In general, exact methods suffer from an exponential runtime scaling.
However, they provide valuable insights in the performance of heuristics and are inevitable for applications which require an optimum exploitation of limited quantum resources.
In this work, we build upon the integer programming model for general qubit routing proposed by Nannicini et al.~\cite{nannicini2021optimalqubitassignmentrouting}.

Several heuristic algorithms have been introduced to address qubit routing where some or all gates are allowed to commute~\cite{Weidenfeller_2022,lao20222qan,matsuo2023sat,jin2023exploiting,kotil2023improvedqubitroutingqaoa}.
A popular method by Weidenfeller et al.~\cite{Weidenfeller_2022} alternatingly applies swap layers and layers of gates that are executable under the current qubit mapping until all gates are scheduled.
Generally, many heuristic tools suffer from large optimality gaps in terms of swaps and depth \cite{tan2020optimality, ping2025assessing}.
Some (near-)exact methods for general qubit routing also adapt to qubit routing where some or all gates are allowed to commute.
The SMT-based method OLSQ, introduced in~\cite{tan2020optimal}, can guarantee optimality for qubit routing with commuting gates with respect to gate count or depth.
Dropping the guarantee of optimality, the authors proposed the faster, near-exact method TB-OLSQ.
Lin et al.~\cite{lin2023scalable} presented the improved methods OLSQ2 (exact) and TB-OLSQ2 (near-exact).
The authors of~\cite{yang2024quantum} and~\cite{shaik2024optimal} further improved upon TB-OLSQ2 using SAT encodings, where \cite{shaik2024optimal} consider circuits consisting only of CNOT gates.
The methods \cite{yang2024quantum, shaik2024optimal, tan2020optimal, lin2023scalable} all solve a series of problems increasing in size.
Our approach also applies this concept.
All (near-)exact methods suffer from high runtimes~\cite{sun2025haqa}.
In this work, we aim to push this boundary for the special case of qubit routing with commuting gates.

\section{Two-Step Optimization Approach}\label{sec:prelim}
In this section, we detail our approach to the qubit routing problem with commuting gates.
In a quantum circuit containing only commuting two-qubit gates, we may aggregate all gates acting on the same pair of qubits into a single two-qubit gate.
This allows us to represent the circuit as an \emph{algorithm graph} $A=(Q,C)$, where the nodes $Q$ represent logical qubits and the edges $C$ represent two-qubit gates.
Also the quantum hardware can be represented by a graph $H=(V,E)$ where the nodes $V$ correspond to physical qubits and the edges $E$ corresponding to available two-qubit gates.
Thus, an instance to the qubit routing problem with commuting gates can be encoded as a pair of graphs $(H,A)$.

\textbf{Full solution.}
For $n\in \N$, we define $[n]\coloneqq\{1,\dots, n\}$.
A routed circuit can be encoded by its initial allocation $f_1:Q \xrightarrow{1:1} V$ of logical qubits to physical qubits
and a sequence of pairs of matchings $\mathcal{F} = ((M_1, C_1),\dots, (M_k, C_k))$, where $M_i, C_i \subset E$ and $M_i\cap C_i=\emptyset$ for all $i\in [k]$.
For each $i\in [k]$, $M_i$ and $C_i$ correspond to the swap gates and algorithm gates in the $i$-th layer of the circuit, respectively.
We require $\bigcup_{i\in [k]} f_i^{-1}(C_i) = C$, where for $i>1$, $f_i:Q \xrightarrow{1:1} V$ is the bijection obtained by applying the swap gates in $M_{i-1}$ to $f_{i-1}$.
More formally, for $i>1$,
\begin{align}\label{eq:mapping}
    f_i \coloneqq \pi_{M_{i-1}} \circ f_{i-1}
\end{align}
where $\pi_{M_{i-1}} \coloneqq \prod_{\{u,v\}\in M_{i-1}} \tau_{uv}$ and $\tau_{uv}:V\xrightarrow{1:1}V$ is the transposition defined by $\{u,v\}\in E$.
We refer to $(f_1,\mathcal{F})$ as a \emph{full solution}.
The depth of the resulting circuit is then exactly $k$.

\textbf{Swap solution.} Since our primary objective is to minimize the number of inserted swap gates, the depth does not have to be encoded explicitly.
Instead, it suffices to encode a solution only by its initial allocation $f_1$ and a sequence of matchings $\mathcal{S} = (M_1,\dots,M_k)$, where $M_i\subset E$ corresponds to the set of swap gates inserted in the circuit at the $i$-th swap layer.
We require that each gate $\{p,q\}\in C$ is \emph{executable}, that is, there exists a $i\in \{1,\dots,k+1\}$ such that $\{f_i(p),f_i(q)\}\in E$, where $f_i$ is defined as above.
We refer to $(f_1,\mathcal{S})$ as a \emph{swap solution}.

We propose a two-stage optimization approach.
The overall solution procedure is illustrated in \Cref{fig:solution_approach}.
\begin{figure}
    \centering
    \includegraphics[page=3,width=1\linewidth]{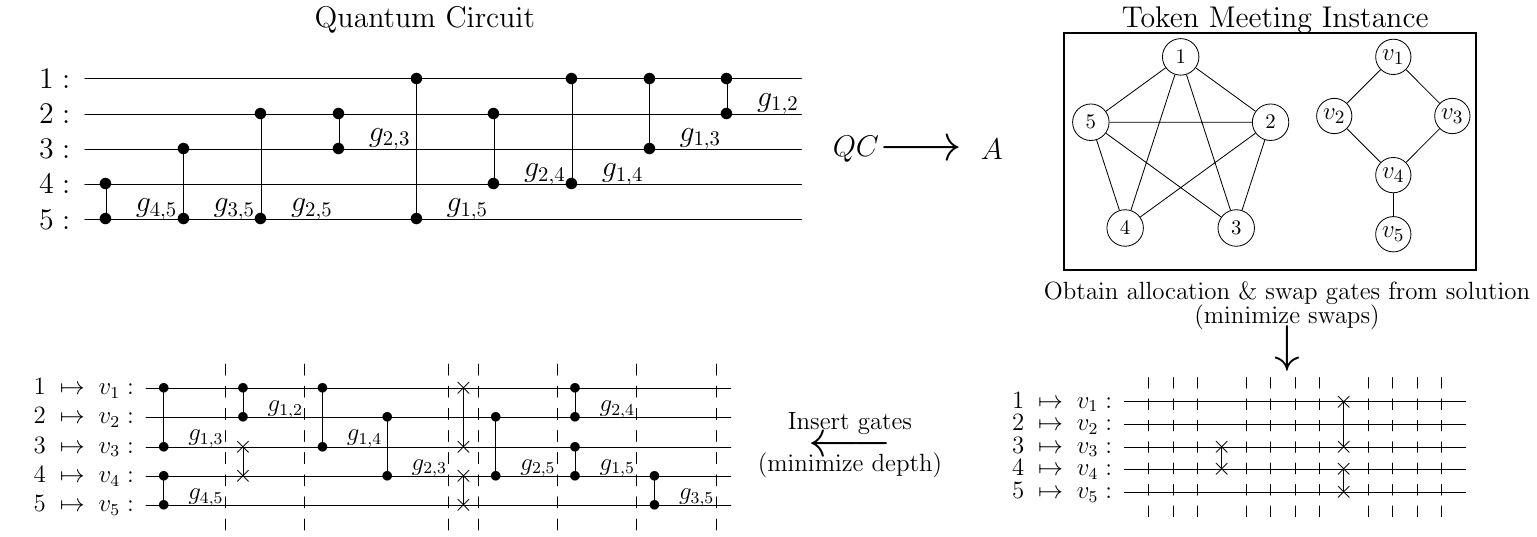}
    \caption{The overall solution approach.
    First, a quantum circuit $QC$ with only commuting two-qubit gates is converted into its graph representation $A$.
    Next, a Token Meeting Instance consisting of the hardware graph and $A$ is solved, which corresponds to finding a swap-optimal swap solution.
    The last step consists of converting the swap solution into a full solution, which can be seen as inserting gates into a circuit while minimizing depth.}
    \label{fig:solution_approach}
\end{figure}
First, we compute an initial mapping of circuit to physical qubits and a set of swap layers ensuring that each gate is executable.
Then, we schedule all gates of the original circuit into either a swap layer or a new layer, aiming to minimize the depth of the resulting circuit.
The advantage of the two-step approach is that a swap solution can be encoded in an integer programming model with less variables than a full solution.
However, we emphasize that the two-step approach is still globally optimal in terms of swap count.
We also employ integer programming to solve the second problem, that is, to compute a full solution with minimum depth from a given swap solution.
This guarantees that the returned full solution has minimum depth among all solutions using the swap layers of our swap solution.
The swap solution representation gives rise to an interesting combinatorial problem, which we call the \emph{Token Meeting Problem}.

\begin{definition}
    An \emph{instance} of the \emph{Token Meeting Problem (TMP)} is given by a pair $(H, A)$ of undirected graphs $H=(V, E)$ and $A=(Q, C)$.
\end{definition}
In the context of the Token Meeting Problem, we refer to the elements of $V$ as \emph{nodes}, $E$ as \emph{edges}, $Q$ as \emph{tokens}, and $C$ as \emph{connections}.

\begin{definition}
    A \emph{solution} $(f_1,\mathcal{S})$ to a TMP instance $(H, A)$ is given by a bijection $f_1:Q \xrightarrow{1:1} V$ and a sequence of matchings $\mathcal{S}=(M_1,\dots,M_k)$, $M_i \subset E$, such that for each connection $\{p,q\}\in C$, the \emph{adjacency condition} is satisfied, that is, there is a $t\in [k+1]$ with $\{f_t(p),f_t(q)\}\in E$.
    Here, $f_t$ for $t> 1$ is defined as in~\eqref{eq:mapping}.
\end{definition}
We define two metrics for evaluating solutions.
\begin{definition}
	Given a solution $(f_1,\mathcal{S})$ to a TMP instance, where $\mathcal{S} = (M_1,\dots,M_k)$, the \emph{number of steps} in the solution is $k$ and the \emph{number of swaps} in the solution is $\sum_{t=1}^{k}|M_t|$.
\end{definition}

\begin{definition}
Given a TMP instance $(H, A)$ and a natural number $t\in \N$, we define $\ML(H, A)$ as the \underline{m}inimum number of \underline{s}waps of any solution, $\MT(H, A)$ as the \underline{m}inimum number of (\underline{t}ime) steps of any solution
and $\ML(H, A, t)$ as the minimum number of swaps of any solution with exactly $t$ steps.
If there is no solution with exactly $t$ steps, we set $\ML(H, A, t)=\infty$.
\end{definition}

We start with a basic observation.
In general, having more steps available cannot increase the number of swaps:
a solution with $t$ steps can be transformed into a solution with more steps by simply adding steps with empty matchings.
The following lemma formalizes this observation.
\begin{lemma}\label{lem:step_increase_decreases_swaps}
	For every TMP instance $(H, A)$ and $t\geq 0$, it holds $\ML(H, A, t+1) \leq \ML(H, A, t)$.
\end{lemma}

Since the inequality of \Cref{lem:step_increase_decreases_swaps} can be strict (see \Cref{ex:strict_ieq}), there is a natural trade-off between minimizing steps and minimizing swaps.
Optimizing for fewer steps might require more total swaps and vice versa.

\begin{example}\label{ex:strict_ieq}
	Let $H=P_6$ be the path graph and let $A$ be the star graph with $6$ nodes, as displayed in \Cref{fig:counterexample_graphs}.
	\begin{figure}
		\centering
        \includegraphics[page=4, width=0.25\textwidth]{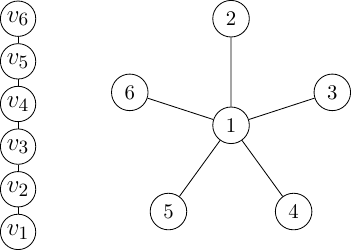}
		\caption{
			An instance to the Token Meeting Problem, defined by the graphs $H$ (left) and $A$ (right).
		}
		\label{fig:counterexample_graphs}
	\end{figure}
	By enumeration, one can verify that $\ML(H, A, 3)= 3 < 4 = \ML(H, A, 2)$ holds.
	A solution $S_1=(f_1,f_2,f_3)$ with four swaps and a solution $S_2=(f'_1,f'_2,f'_3,f'_4)$ with three swaps are shown in \Cref{fig:counterexample_sol1}.
	\begin{figure}
		\centering
        \includegraphics[page=5, width=0.6\textwidth]{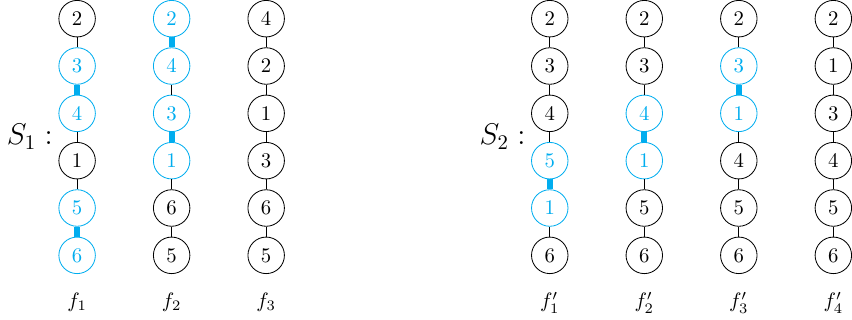}
		\caption{
			Two solutions $S_1$ (left) and $S_2$ (right) to the TMP instance of \Cref{fig:counterexample_graphs}.
			For each step, the bijection from $Q$ to $V$ is displayed, and edges corresponding to swaps are highlighted in blue.
			Although $S_2$ needs more steps than $S_1$, it requires fewer swaps.
		}
		\label{fig:counterexample_sol1}
	\end{figure}
\end{example}

\subsection{Computing a Swap Optimal Solution}
In the following, we describe an algorithm to compute $\ML(H, A)$ given the ability to compute $\ML(H, A, t)$ for $t\in \N$.
The motivation is that we propose several compact integer programming models to compute $\ML(H, A, t)$ in \Cref{sec:ip_models}, while it is unlikely that a compact integer programming model for $\ML(H, A)$ even exists, without having prior knowledge of $\ML(H, A)$.
\Cref{alg:step_increase} summarizes our approach to compute $\ML(H, A)$.
\begin{algorithm}
	\caption{}\label{alg:step_increase}
	\begin{algorithmic}[1]
		\Require A TMP instance $(H, A)$.
		\Ensure $\ML(H, A)$
		\For{$t=0,1,\dots$}
		\If{$\ML(H, A, t)=t$} 
		\Return $t$.
		\EndIf
		\EndFor
	\end{algorithmic}
\end{algorithm}

In simple words, we iteratively compute $\ML(H, A, t)$ for $t=0,1,\dots$
until $t=\ML(H, A, t)$.
Then, it holds $t=\ML(H, A)$ as the following lemma shows.
\begin{lemma}\label{prop:termination_increasing_steps}
	Let $(H, A)$ be a TMP instance.
	Then $\ML(H, A, t)=t$ holds if and only if $t=\ML(H, A)$.
\end{lemma}
\begin{proof}
	We prove the statement by showing that the sequence $(\ML(H, A, t)-t)_{t\geq\MT(H, A)}$ is strictly decreasing with $\ML(H, A, t)-t=0$ for $t=\ML(H, A)$.
	From \Cref{lem:step_increase_decreases_swaps}, we have
	\[
		\ML(H, A, t)-t\geq \ML(H, A, t+1)-t>\ML(H, A, t+1)-(t+1)
	\]
	for all $t\geq \MT(H, A)$, which shows that the sequence is strictly decreasing.
	
	Let $S$ be a solution with $t=\ML(H, A)$ swaps and only non-empty matchings.
	By definition, we have $t=\ML(H, A)\leq \ML(H, A, t)$.
	By potentially adding steps with empty matchings to $S$, we can construct a solution with exactly $t$ steps and $\ML(H, A)$ swaps.
	Therefore, $\ML(H, A, t)=\ML(H, A)=t$.
\end{proof}

Bounds on $\ML(H,A)$ and $\MT(H, A)$ are useful to \Cref{alg:step_increase} in several ways.
First, an \emph{upper bound} $t^*$ on $\ML(H, A)$ directly yields an upper bound on the number of steps considered in \Cref{alg:step_increase}, since $\ML(H, A, t^*)=\ML(H, A)$.
However, if $t^*$ is significantly larger than $\ML(H, A)$, computing $\ML(H, A, t^*)$ might be computationally more expensive than \Cref{alg:step_increase}.
Conversely, a nontrivial \emph{lower bound} on $\MT(H, A)$ can help avoid unnecessary computation in early iterations of \Cref{alg:step_increase} by skipping values of $t$ with $t<\MT(H, A)$ where $\ML(H, A, t)=\infty$ holds.

\section{Bounds on the Number of Swaps and Steps}
\label{sec:bounds}
In this section, we derive upper and lower bounds on $\MT(H,A)$ and $\ML(H,A)$.
We use these bounds to prove an asymptotic worst-case scaling of $\MT(H,A)$ and $\ML(H,A)$ and to avoid unnecessary computations in \Cref{alg:step_increase}.
First, we collect a preliminary result.

\begin{lemma}\label{lem:mt_ml_bounds}
If $(H, A)$ is a TMP instance, then $\MT(H, A)\leq \ML(H, A)\leq \lfloor n/2 \rfloor \cdot \MT(H, A)$.
\end{lemma}
\begin{proof}
	It is easy to see that $\MT(H, A)\leq \ML(H, A)$.
	Since the size of a matching of $H$ is bounded by $\lfloor n/2 \rfloor$, at most $\lfloor n/2 \rfloor$ swaps can occur per step, implying that a solution with $\MT(H, A)$ steps has at most $\lfloor n/2 \rfloor \cdot \MT(H, A)$ swaps.
\end{proof}

\subsection{Upper Bounds}
In order to derive upper bounds on $\MT(H,A)$ and $\ML(H,A)$, we will show that it suffices to derive such bounds only for a particular class of extreme instances.
First we observe that, given a TMP instance $(H, A)$, replacing $H$ or $A$ with an isomorphic graph essentially preserves the instance.
We formalize this in the following lemma, which can be proven with elementary manipulations.
\begin{lemma}\label{cor:isomorphism_trafo}
    Let $(H, A)$ be an TMP instance and let $(f_1,\dots, f_{k+1})$ be the sequence of bijections of a solution $S=(f_1,(M_1,\dots, M_k))$.
    If $H$ is isomorphic to the graph $H'=(V',E')$ with the isomorphism $\sigma_{H}: V\mapsto V'$ and $A$ is isomorphic to the graph $A'=(Q',C')$ with the isomorphism $\sigma_{A}: Q\mapsto Q'$, then $(\sigma_{H} \circ f_1 \circ \sigma_{A},\dots, \sigma_{H} \circ f_{k+1}\circ \sigma_{A})$ corresponds to a solution to the TMP instance $(H',A')$ with the same number of swaps and steps as $S$.
\end{lemma}
Given a solution $S$ to a TMP instance $(H, A)$, deleting connections from $A$ means that the adjacency condition remains satisfied for all remaining connections.
Also, by adding edges to $H$, the matchings $(M_1,\dots, M_k)$ remain matchings in the modified graph.
Informally, $S$ is also a solution to all instances that are obtained by making $A$ ``smaller'' and $H$ ``larger''.
\begin{lemma}\label{lem:subgraph_hierarchy}
Let $(H, A)$ and $(H', A')$ be TMP instances on $n$ nodes, where $A'$ is isomorphic to a spanning subgraph of $A$, and $H$ is isomorphic to a spanning subgraph of $H'$. Then
\begin{equation*}
    \MT(H', A')\leq \MT(H, A)\quad\text{ and }\quad \ML(H', A', t)\leq \ML(H, A, t)\quad \text{ for all } t\in \N.
\end{equation*}
\end{lemma}
\begin{proof}
    It suffices to show that any solution to $(H, A)$ can be transformed into a solution to $(H', A')$ with the same number of steps and swaps.
    By assumption, $(H, A)$ is isomorphic to an instance $(\bar H, \bar A)$, where $\bar H$ is a subgraph of $H'$ and $A'$ is a subgraph of $\bar A$.
    Applying \Cref{cor:isomorphism_trafo} yields a solution to $(\bar H, \bar A)$ for any solution to $( H, A)$.
    The lemma now follows by noting that any solution to $(\bar H, \bar A)$ is also a solution to $(H', A')$.
\end{proof}
The natural extreme cases of \Cref{lem:subgraph_hierarchy} are the instances where $H$ is a tree and $A$ is the complete graph $K_n$.
In particular, to state upper bounds on the minimum number of swaps and steps of any instance on $n$ nodes, it suffices to derive such bounds only for the set of extreme instances $(H, A)$ where $H$ is a tree and $A=K_n$.

To construct an upper bound, we provide a constructive algorithm that returns solutions with few swaps.
We briefly explain the main idea behind the algorithm.
First, pick a node $v\in V$, perform a depth first search (DFS) starting from $v$, and then swap the token $p\in Q$ placed on $v$ along edges in the order traversed by the DFS algorithm.
The token $p$ becomes adjacent to all other tokens with this swap sequence.
Since the last node visited by DFS is a leaf, we can ``delete'' the leaf holding token $p$ from the remaining tree.
We can construct a solution by iterating until the tree has only two remaining nodes.
Note that DFS traverses each edge in the tree at most twice (see e.g. \cite{bondy2008graph}), which allows us to bound the number of swaps in the constructed solution.
A formal version of this algorithm is given in \Cref{alg:swap_algorithm}.

\begin{algorithm}
    \caption{DFS Swap Algorithm}\label{alg:swap_algorithm}
    \begin{algorithmic}[1]
        \Require An undirected tree $H$ with $n$ nodes.
        \Ensure A solution $S$ to the TMP instance $(H, K_n)$.
        \State Choose an arbitrary initial bijection $f_1$ and set $T \leftarrow H$.
        \For{$t=1,\dots, n-2$}
        \State Select a node $v_t$ holding the token $q_t$ of $T$ neighboring a leaf $l_t$ of $T$.
        \State \parbox[t]{\dimexpr\linewidth-\algorithmicindent}{
        Perform DFS on $T$ starting from $l_t$ and let $\ell: V(T) \mapsto [n-t+1]$ be the function where $\ell(v)$ is the number of nodes that were reached by DFS prior to $l_t$ plus one.
        }
        \For{$k=3,\dots, n-t+1$}
        \If{$\ell^{-1}(k)$ is not a leaf}
        \State Move the token $q_t$ to the node $\ell^{-1}(k)$ with swaps along the unique path in $T$.
        \EndIf
        \State $k \leftarrow k + 1$
        \EndFor
        \State Move the token $q_t$ to the adjacent leaf node $\ell^{-1}(n-t+1)$.
        \State $T\leftarrow T[V(T)\setminus \{\ell^{-1}(n-t+1)\}]$
        \EndFor
        \Return the constructed solution.
    \end{algorithmic}
\end{algorithm}
\begin{lemma}\label{thm:alg_swap_count}
For every undirected tree $H$ with $n$ nodes, \Cref{alg:swap_algorithm} returns a solution to the TMP instance $(H, K_n)$ with at most $(n-2)^2$ swaps.
\end{lemma}
\begin{proof}
    Without loss of generality, let $n\geq 3$.
    First, we will prove the correctness of \Cref{alg:swap_algorithm}.
    Consider step $t=1$.
    With lines 5-8, token $q_1$ becomes adjacent to every other token, and all connections with $q_1$ satisfy the adjacency condition.
    Note that line 9 is correct since the last node visited by DFS is a leaf.
    By induction on the number of nodes in the tree $T$, all connections with token $q_1,\dots,q_{n-2}$ satisfy the adjacency condition, and the remaining connection satisfies the adjacency condition in the last step.

    For all $t\in[n-2]$, each of the $n-t$ edges of $T$ is traversed at most twice.
    The edge connecting $v_t$ with the leaf $l_t$ is not traversed at all, and the edge used to move $q_t$ to the final leaf is traversed only once.
    Thus, a bound on the number of swaps of the returned solution is given by $\sum_{t=1}^{n-2}(2(n-t-1)-1)= (n-2)^2$.
\end{proof}
\begin{figure}
\centering
\includegraphics[page=6, width=0.7\textwidth]{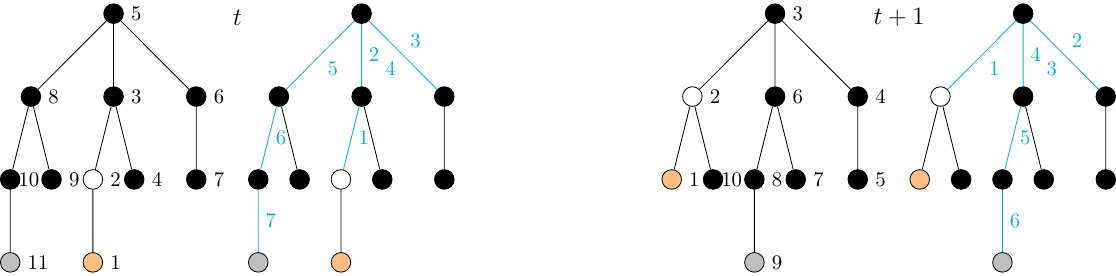}
\caption{
    Illustration of two steps of \Cref{alg:swap_algorithm}.
    For each step $t$, the nodes $v_t$ (white), $l_t$ (orange) and the node labels $\ell$ are shown on the left.
    Each edge involved in a swap is highlighted in blue, with labels denoting the order in which the edge swaps are applied to move $q_t$ from $v_t$ to the end node (gray).
}
\label{fig:quadratic_swap_example}
\end{figure}
As a corollary of \Cref{thm:alg_swap_count}, we obtain the first main result of this section.
\begin{theorem}\label{result:quad_upper_bound}
    Every TMP instance $(H,A)$ has a solution with $\mathcal{O}(n^2)$ swaps and $\mathcal{O}(n^2)$ steps.
\end{theorem}
An example for an iteration of \Cref{alg:swap_algorithm} is given in \Cref{fig:quadratic_swap_example}.
We note that \Cref{alg:swap_algorithm} is inherently sequential in the sense that it constructs a solution with exactly one swap per step.
We further note that \Cref{alg:swap_algorithm} is not optimized for practical performance.
The practical performance could be improved, for example, by reducing the number of steps in the solution by swapping two tokens along the DFS path at the same time.
However, this would not improve the asymptotics of \Cref{alg:swap_algorithm}.

\subsection{Lower Bounds}
In this subsection, we prove lower bounds on the number of swaps and steps for several classes of TMP instances.
To obtain a lower bound on the number of swaps of a solution, we first define a key parameter.
\begin{definition}\label{def:bound_swaps}
    For a graph $H=(V, E)$, we define
    \begin{equation*}
        \Delta'(H):= \max \{|N(i) \cup N(j)|-|N(i)\cap N(j)|-2: \{i,j\} \in E\},
    \end{equation*}
    where $N(i)\subseteq V$ denotes the set of neighbors of node $i\in V$.
\end{definition}
It follows immediately that
\[
\Delta'(H)\leq \max \{\deg(i) + \deg(j) - 2: \{i,j\} \in E\} \leq 2\Delta(H)-2,
\]
where $\deg(i)$ denotes the degree of node $i$ and $\Delta(H)$ is the maximum degree of $H$.
We obtain the following lower bound on the number of swaps.
\begin{lemma}\label{lem:lower_swap_bound}
    For every TMP instance $(H, A)$, we have
    \begin{equation*}
        \ML(H, A) \geq \dfrac{|C|-|E|}{\Delta'(H)}.
    \end{equation*}
\end{lemma}
\begin{proof}
    Fix an arbitrary solution $S=(f_1,(M_1,\dots, M_k))$.
    With the first bijection of $S$, at most $|E|$ connections satisfy the adjacency condition.
    Now, we show that at most $\Delta'(H)$ connections can newly satisfy the adjacency condition with a single swap.
    Observe that swapping a token $p$ from node $i$ to $j$ and token $q$ from node $j$ to $i$ makes $p$ adjacent to the tokens placed on the neighbors $N(j)$ of $j$, and $q$ to the tokens placed on the neighbors $N(i)$ of $i$.
    Since the connection $\{p,q\}$ and all connections $\{p,r\},\{q,r\}$ for tokens $r$ placed on nodes in $N(i)\cap N(j)$ satisfied the adjacency condition before the swap, at most $|N(i)|+|N(j)|-2|N(i)\cap N(j)|-2\leq \Delta'(H)$ additional connections satisfy the adjacency condition after the swap.
    Thus, $S$ has at least $(|C|-|E|)/\Delta'(H)$ swaps.
\end{proof}
We also obtain a lower bound on the number of steps if $A$ is the complete graph on $n$ nodes.

\begin{proposition}\label{lem:lower_bounds_by_max_degree}
    Let $(H_n)_{n\in \N_{\geq 1}}$ be a family of connected graphs such that $H_n$ has maximum degree $\Delta_{n}$.
    Then,
    \begin{align*}
        \ML(H_n, K_n)\geq \dfrac{n(n-\Delta_n-1)}{4\Delta_n - 4} \quad\text{ and }\quad \MT(H_n, K_n)\geq \dfrac{n-\Delta_n-1}{2\Delta_n - 2}.
    \end{align*}
\end{proposition}
\begin{proof}
    The lower bound on $\ML(H_n, K_n)$ follows from applying \Cref{lem:lower_swap_bound} to the instance $(H_n, K_n)$ and using $|E(H_n)|\leq n \Delta_n/2$, $|E(K_n)| = n(n-1)/2$ and $\Delta'(H_n)\leq 2\Delta_n-2$.
    The bound on $\MT(H_n, K_n)$ then follows from the inequality $\MT(H, K_n)\geq \ML(H, K_n) \cdot \frac{2}{n}$ of \Cref{lem:mt_ml_bounds}.
\end{proof}

Using \Cref{lem:lower_bounds_by_max_degree}, we obtain our second main result, \Cref{result:existence_of_quad_swap_instances}, which shows that the upper bound of $\mathcal{O}(n^2)$ swaps in \Cref{result:quad_upper_bound} is asymptotically tight.
\begin{theorem}\label{result:existence_of_quad_swap_instances}
There are families of connected graphs $(H_n)_{n\in \N_{\geq 1}}$ with $\ML(H_n, K_n)\in \Omega(n^2)$ and $\MT(H_n, K_n)\in \Omega(n)$.
\end{theorem}
\begin{proof}
    Consider any family of connected graphs where all members have a maximum degree bounded by a constant, such as the family of path graphs. \Cref{lem:lower_bounds_by_max_degree} then implies the statement.
\end{proof}

In \Cref{lem:lower_bounds_by_max_degree}, we constructed a lower bound on $\MT(H, K_n)$ based on a lower bound on the minimum number of swaps $\ML(H, K_n)$.
However, it is also possible to construct a lower bound on $\MT(H, A)$ directly, based on a similar reasoning as in \Cref{lem:lower_swap_bound}.
\begin{lemma}\label{lem:lower_step_bound}
    For every TMP instance $(H, A)$, we have
    \begin{equation*}
        \MT(H, A) \geq \dfrac{|C|-|E|}{D(H)},
    \end{equation*}
    where
    \begin{equation*}
        D(H)=\max_{K\in [\lfloor n/2\rfloor]}\left(\min\left(\sum_{k=1}^{K}(d_{2k-1} + d_{2k} - 2), |E|-K\right)\right)
    \end{equation*}
    and $d_1,\dots, d_n$ is the non-increasing degree sequence of $H$.
\end{lemma}
\begin{proof}
    Let $S$ be an arbitrary solution.
    With the first bijection $f_1$ of $S$, at most $|E|$ connections satisfy the adjacency condition.
    As in \Cref{lem:lower_swap_bound}, we now construct a bound on the number of connections that can newly satisfy the adjacency condition with a single step.
    Let $M$ be a matching of the solution $S$ and let $\{i,j\}\in M$.
    Let $p$ be the token that moves from node $i$ to $j$ and is swapped with the token $q$.
    Then, unlike in the proof of \Cref{lem:lower_swap_bound}, where we considered the neighbors of $p$ after only a single swap, there are now $\deg(j)-1$ candidate tokens that could be newly adjacent to the token $p$, since all nodes in $N(i)\cap N(j)$ can also be involved in a swap.
    Similarly, token $q$ is an element of at most $\deg(i)-1$ newly adjacent connections.
    As any newly adjacent connection must contain a token which is contained in the matching $M$ and each token-pair that is placed on an edge of the matching cannot be newly adjacent, we obtain that at most
    \begin{align*}
        \min\{\sum_{\{i,j\} \in M}(\deg(i)+\deg(j)-2), |E|-|M|\}
    \end{align*}
    connections newly satisfy the adjacency condition.
    With $d_1,\dots, d_n$ being the non-increasing degree sequence of $H$, we obtain that in a single step at most 
    \begin{align*}
        &\max\{\min\{\sum_{\{i,j\} \in M}(\deg(i)+\deg(j)-2), |E|-|M|\}: M \text{ is a matching of }H\}\\
        \leq
        &\max_{K\in [\lfloor n/2\rfloor]}\left(\min\left(\sum_{k=1}^{K}(d_{2k-1} + d_{2k} - 2), |E|-K\right)\right)=D(H)
    \end{align*}
    connections newly satisfy the adjacency condition.
    Thus, $S$ has at least $(|C|-|E|)/D(H)$ steps.
\end{proof}
We use \Cref{lem:lower_step_bound} to extend a recent result by Weidenfeller et al.~\cite{Weidenfeller_2022} whose result implies that $\MT(P_n,K_n)=n-2$, where $P_n$ is the path graph on $n$ nodes.
We show that $\Omega(n)$ many steps are necessary for any instance $(H, K_n)$, where $H$ is a tree on $n$ nodes.
\begin{corollary}\label{cor:lower_step_bound_trees}
    Let $H=(V, E)$ be a tree on $n$ nodes.
    Then
    \begin{equation*}
        \MT(H, K_n)\geq \dfrac{n-1}{2}.
    \end{equation*}
\end{corollary}
\begin{proof}
    We have $D(H)\leq |E|-1= n-2$ and $|E(K_n)| = n(n-1)/2$.
    By \Cref{lem:lower_step_bound}, we have
    \begin{equation*}
        \MT(H, K_n)
        \geq \dfrac{|E(K_n)|-|E|}{D(H)}
        \geq \dfrac{|E(K_n)|-|E|}{|E|-1}
        \geq \dfrac{n(n-1)-2(n-1)}{2(n-2)}
        =\dfrac{n-1}{2}.
    \end{equation*}
\end{proof}
A natural question arising from \Cref{cor:lower_step_bound_trees} is whether this bound is asymptotically tight, that is, whether $\MT(H, K_n)\in \mathcal{O}(n)$ holds if $H$ is a tree.
Weidenfeller et al.~\cite{Weidenfeller_2022} showed that this is the case for path graphs.
It is not hard to see that this is also true for star graphs:
We have $\MT(K_{1,n-1}, K_n)=n-2$, where $K_{1,n-1}$ is the star graph on $n$ nodes.

It remains an open question whether or not $\MT(H_n, K_n)\in \Theta(n)$ holds for all families of trees $(H_n)_{n\in \N_{\geq 1}}$.
Computational experiments reveal that $\MT(H_n, K_n)\in\{n-2,n-1\}$ holds for all trees with up to 10 nodes.
Therefore, we conjecture the following.
\begin{conjecture}\label{conjecture:tree_steps}
    We have $\MT(H_n, K_n)\in \Theta(n)$ for every family of trees $(H_n)_{n\in \N_{\geq 1}}$.
\end{conjecture}

A \emph{quadratic subdivided star} $\text{QSST}_n$ is a tree defined for values of $n$ where there exists a $m\in \N$ with $n=m^2+1$.
$\text{QSST}_n$ consists of a root node that has $m$ added paths of length $m$.
In our computational experiments, the only tree with $n=9$ nodes that required $n-1=8$ steps was a quadratic subdivided star with one leaf removed.
Thus, considering this instance class might be useful for proving or disproving \Cref{conjecture:tree_steps}.
Although it remains an open problem whether $\mathcal{O}(n)$ steps suffice to solve the TMP instance class $(\text{QSST}_n, K_n)_{n\in\mathbb{N}}$, we improve the trivial upper bound of $\mathcal{O}(n^2)$ steps of \Cref{result:quad_upper_bound} to $\mathcal{O}(n^{3/2})$ steps.
\begin{theorem}
    \begin{equation*}
        \MT(\text{\normalfont QSST}_n,K_n)\in\mathcal{O}(n^{3/2})
    \end{equation*}
\end{theorem}
\begin{proof}
    We give an explicit solution to $(\text{QSST}_n,K_n)$ with $\mathcal{O}(n^{3/2})$ steps.
    Let $n=m^2+1$.
    We use the following observation:
    Given an arbitrary placement of $k$ tokens on the path graph with $k$ nodes, there is a sequence of $2k-3$ steps such that each token is placed on the upper leaf at least once, and the token previously placed on the upper leaf is placed on the bottom leaf after the sequence.
    We refer to this sequence as a \emph{path sequence}.
    See \Cref{fig:sequence_illustration} for an illustration of a path sequence.
    After an arbitrary initial placement that places some token $q_1$ on the root node, perform the path sequence for all paths in parallel.
    We refer to this as one \emph{iteration}.
    The first iteration requires $2(m-1)-3$ steps and guarantees that all tokens have been placed on a node adjacent to the root node.
    Therefore, all connections with token $q_1$ now satisfy the adjacency condition.
	Next, swap the token $q_1$ into one of the longest remaining paths and perform the next iteration.
	After the second iteration, the token $q_1$ is placed on a leaf and can be deleted from the tree.
	After $m$ iterations, the length of the longest remaining path is reduced by one.
    Inductively, by iterating until the length of the longest remaining path is one, we obtain a solution with $\sum_{k=2}^{m} m\cdot 2(k-1) = (m - 1)m^2\in \mathcal{O}(m^3)$ steps.
    
    \begin{figure}
    	\centering
        \includegraphics[page=7, width=0.325\textwidth]{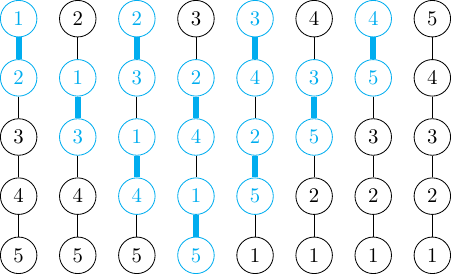}
    	\caption{
    		A sequence of 7 steps where each token is placed on the upper leaf at least once and the token 1, which was placed on the upper leaf at the start, is placed on the bottom leaf after the sequence.
    		For each step, the edges in the matching are highlighted in blue.
    	}
    	\label{fig:sequence_illustration}
    \end{figure}
\end{proof}

\subsection{Complexity}
Ito et al.~\cite{ito2023algorithmic} study the problem of qubit routing in a more general setting that includes both general qubit routing and qubit routing with commuting gates.
They prove that several special cases of qubit routing are $\mathcal{NP}$-hard.
We extend their results by showing that qubit routing with commuting gates is $\mathcal{NP}$-hard as well.
To this end, we consider the following decision problems.
\begin{problem}\textsc{Parallel Token Meeting Problem (PTMP)}\\
	\textbf{Input:} A TMP instance $(H, A)$ and a nonnegative integer $k\in \N$.\\
	\textbf{Question:} Is there a solution with at most $k$ steps?
\end{problem}
\begin{problem}\textsc{Sequential Token Meeting Problem (STMP)}\\
	\textbf{Input:} A TMP instance $(H, A)$ and a nonnegative integer $k\in \N$.\\
	\textbf{Question:} Is there a solution with at most $k$ swaps?
\end{problem}
It is easy to see that the PTMP and STMP are in $\mathcal{NP}$, since an optimal solution,
whose size is bounded polynomially in the number of nodes of the instance,
serves as a polynomial certificate for a yes-instance.
Additionally, by reducing from the subgraph isomorphism problem, we can show $\mathcal{NP}$-completeness for PTMP and STMP, which implies $\mathcal{NP}$-hardness of the optimization versions.
\begin{theorem}\label{thm:NPcompleteness}
	The \textsc{Sequential Token Meeting Problem} and the \textsc{Parallel Token Meeting Problem} are $\mathcal{NP}$-complete.
\end{theorem}

A proof of \Cref{thm:NPcompleteness} is given in \Cref{app:complexity}.
We note here that we prove the hardness of both problems using the ``hardness'' of choosing the optimal first bijection.
It is an interesting open question whether Token Meeting remains $\mathcal{NP}$-hard if one fixes the first bijection.

\section{Integer Programming Models}
\label{sec:ip_models}

$\mathcal{NP}$-hard optimization problems are often solved by mixed-integer programming methods which can solve even large instances to proven optimality~\cite{Koch2022}.
In this section, we derive an integer programming model for the TMP whose feasible set is the set of all solutions with a fixed number of steps $T\in \N$.
Thus, the integer programming model enables us to compute $\ML(H, A,T)$.
However, the model has constraints with products of binary variables, which makes it nonlinear.
In order to use established solution techniques for linear integer programming, we present several ways to linearize our model.
Furthermore, in \Cref{sec:polytopes}, we strengthen our model by deriving linear descriptions of certain sub-polytopes present in our models.
Our computational experiments show that these linear descriptions speeds up the solution time by a factor of up to 5.

\subsection{A Network Flow Model}
Our model is based on the work by Nannicini et al.~\cite{nannicini2021optimalqubitassignmentrouting}, who study a closely related problem.
We modify their model to fit our purposes.
We define the directed edge set $\bar{E}$, which contains two arcs $(i, j)$ and $(j, i)$ for each undirected edge $\{i, j\}\in E$.
Let $\bar{N}(i)\define N(i)\cup \{i\}$ denote the set of neighbors of a node $i\in V$ including $i$ itself.
For the sake of shorter notation, we represent each connection $(p,q)$ as an ordered pair instead of an unordered set $\{p,q\}$.

The model includes binary \emph{location} variables $w^t_{pi}\in\{0,1\}$ for all $p\in Q,\, i\in V,\,t\in [T]$, where $w^t_{pi}=1$ indicates that token $p$ is located at node $i$ in bijection $f_t$.
Additionally, it includes binary \emph{routing} variables $x^t_{pij}\in\{0,1\}$ for all $p\in Q,\, i\in V,\, j\in \bar{N}(i),\, t\in [T-1]$, where $x^t_{pij}=1$ indicates that token $p$ moves from node $i$ to node $j$ in step $t$.
The following model characterizes the set of solutions with exactly $T-1$ steps:

\begin{subequations}\label{model:1model}
\begin{align}
    \min \quad&\frac{1}{2}\sum_{t\in [T-1]}\sum_{p\in Q}\sum_{(i, j)\in \bar{E}}x^t_{pij}
    \label{model:1model_obj}\\
    &\sum_{p\in Q} w^t_{pi} = 1,\quad && i\in V,\, t\in [T], \label{model:1model_assignqubit2node}\\
    &\sum_{i\in V} w^t_{pi} = 1,\quad && p\in Q,\, t\in [T], \label{model:1model_assignnode2qubit}\\
    &\sum_{j\in \bar{N}(i)}x^t_{pij}= w^t_{pi},\quad&& p\in Q,\,i\in V,\, t\in [T-1],
    \label{model:1model_wx_convhull1}\\
    &\sum_{j\in \bar{N}(i)}x^{t}_{pji}= w^{t+1}_{pi},\quad&& p\in Q,\,i\in V,\, t\in [T-1],
    \label{model:1model_wx_convhull2}\\
    &\sum_{p\in Q}x^t_{pij} = \sum_{p\in Q}x^t_{pji},\quad&& \{i, j\} \in E, \, t\in [T-1],
    \label{model:1model_xcycle}\\
    &\sum_{t\in [T]}\sum_{(i,j)\in \bar{E}} w^t_{pi} \cdot w^t_{qj}\geq 1,\quad&&  (p,q)\in C,
    \label{model:1model_gate_bilinear}\\
    &w^t_{pi}\in\{0,1\}, \quad&&  p\in Q,\, i\in V,\,t\in [T],
    \label{model:1model_wvar}\\
    &x^t_{pij}\in\{0,1\}, \quad& & p\in Q,\, i\in V,\, j\in \bar{N}(i),\, t\in [T-1],
    \label{model:1model_xvar}
\end{align}
\end{subequations}
The constraints \eqref{model:1model_assignqubit2node} and \eqref{model:1model_assignnode2qubit} establish a one-to-one mapping between tokens and nodes at each step.
Constraints \eqref{model:1model_wx_convhull1} and \eqref{model:1model_wx_convhull2} connect the $x$ and $w$ variables and ensure that a token is placed on adjacent nodes in subsequent steps.
Constraints \eqref{model:1model_xcycle} ensure that at any given step, if a token moves from node $i$ to the adjacent node $j$ such that $i\neq j$, there is another token which moves from $j$ to $i$.
Constraints \eqref{model:1model_gate_bilinear} ensure that all adjacent tokens are placed on adjacent nodes in some step.
The objective \eqref{model:1model_obj} models the number of swaps in a solution.
Finally, \eqref{model:1model_wvar} and \eqref{model:1model_xvar} define the domains of the variables.

Model~\ref{model:1model} can be interpreted as a network flow problem on a time-expanded version $H^{\rm ext}$ of $H$,
as illustrated in \Cref{fig:example_time_exanded_graph}, with additional constraints.
The vertices of $H^{\rm ext}$ consist of $T$ copies of $V$ and arcs connect each node in one copy of $V$ with the copy of itself and its neighbors in the next copy of $V$.
Thus, a solution corresponds to $n$ node-disjoint paths from sources to sinks in $H^{\rm ext}$, subject to additional constraints.

\begin{figure}
    \centering
    \includegraphics[page=8, width=0.5\textwidth]{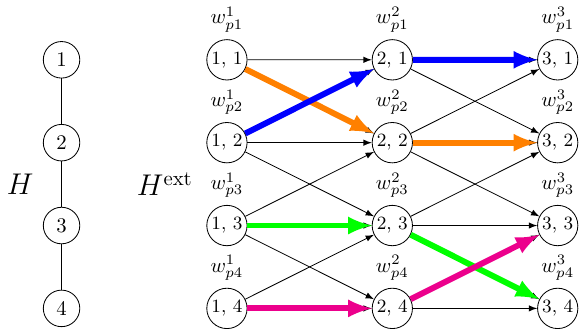}
    \caption{
    An example for the time-expanded hardware $H^{\rm ext}$ graph for $T=3$ and $H=P_4$. 
    A solution of a TMP instance corresponds to a set of $n$ disjoint paths in $H^{\rm ext}$.
    For each token $p$, the variable $w^t_{pi}$ has value 1 if and only if $p$ is placed on node $i$ after $t-1$ steps.}
\label{fig:example_time_exanded_graph} 
\end{figure}

\subsection{Strengthened McCormick Linearization}
To enable the use of integer linear programming solvers, we need to linearize the bilinear constraints \eqref{model:1model_gate_bilinear}.
A straightforward method is to introduce binary \emph{connection} variables $y^t_{pqij}\in \{0,1\}$ for all $(p,q)\in C,\, (i,j)\in \bar{E},\, t\in [T]$, which model the product $w^t_{pi}\cdot w^t_{qj}$.
We then replace constraints~\eqref{model:1model_gate_bilinear} by $\sum_{t\in [T]}\sum_{(i,j)\in \bar{E}} y^t_{pqij}\geq 1$ for all connections $(p, q)$.
Linearizing $y^t_{pqij}=w^t_{pi}\cdot w^t_{qj}$ with the standard McCormick constraints as in \cite{nannicini2021optimalqubitassignmentrouting} leads to the following model:
\begin{subequations}
\begin{align}
    \hypertarget{model:MCY}{\text{YEQ}}: \min \quad&\frac{1}{2}\sum_{t\in [T-1]}\sum_{p\in Q}\sum_{(i, j)\in \bar{E}}x^t_{pij}\\
    &\eqref{model:1model_assignqubit2node}-\eqref{model:1model_xcycle}, \eqref{model:1model_wvar}-\eqref{model:1model_xvar},&&\\
    &\sum_{t\in [T]}\sum_{(i,j)\in \bar{E}} y^t_{pqij}\geq 1,\quad&&  (p,q)\in C,
    \label{model:1amodel_gate_y}\\
    &y^t_{pqij}\leq w^t_{pi}, \quad&&  (p,q)\in C,\, (i,j)\in \bar{E},\,t\in [T],
    \label{model:1amodel_mccormick1}\\
    &y^t_{pqij}\leq w^t_{qj}, \quad&& (p,q)\in C,\, (i,j)\in \bar{E},\,t\in [T],
    \label{model:1amodel_mccormick2}\\
    &y^t_{pqij}\geq w^t_{pi} + w^t_{qj} - 1, \quad&&  (p,q)\in C,\, (i,j)\in \bar{E},\,t\in [T],
    \label{model:1amodel_mccormick3}\\
    &y^t_{pqij}\in \{0,1\}, \quad&&  (p,q)\in C,\, (i,j)\in \bar{E},\, t\in [T].
    \label{model:1amodel_yvar}
\end{align}
\end{subequations}
In our case, modeling only the inequality $y^t_{pqij}\leq w^t_{pi} \cdot w^t_{qj}$ suffices, since constraints \eqref{model:1amodel_gate_y} ensure that for each connection $(p,q)$, there exists a $t \in [T]$ and an edge $(i,j)\in \bar{E}$ such that $y^t_{pqij}=1$.
Although the inequality $y^t_{pqij}\leq w^t_{pi} \cdot w^t_{qj}$ could be modeled by simply dropping the McCormick constraints~\eqref{model:1amodel_mccormick3}, McCormick constraints usually result in a weak linearization
if additional constraints are present~\cite{Boland_2016}.
We obtain a stronger linearization by multiplying constraints~\eqref{model:1model_assignnode2qubit} with variables $w^t_{qj}$ and removing bilinear terms that are zero by assumption, similar to the Reformulation-Linearization Technique~\cite{sherali1992new}.
This linearization results in the following strengthened model:
\begin{subequations}
\begin{align}
    \hypertarget{model:RLTY}{\text{YIEQ}}: \min \quad&\frac{1}{2}\sum_{t\in [T-1]}\sum_{p\in Q}\sum_{(i, j)\in \bar{E}}x^t_{pij}\\
    &\eqref{model:1model_assignqubit2node}-\eqref{model:1model_xcycle}, \eqref{model:1model_wvar}-\eqref{model:1model_xvar}, \eqref{model:1amodel_gate_y},\eqref{model:1amodel_yvar},&&\\
    &\sum_{j\in N(i)}y^t_{pqij}\leq w^t_{pi}, \quad&& (p,q)\in C,\, i\in V,\,t\in [T],
    \label{model:1cmodel_rlt1}\\
    &\sum_{j\in N(i)}y^t_{pqji}\leq w^t_{qi}, \quad&& (p,q)\in C,\, i\in V,\,t\in [T].
    \label{model:1cmodel_rlt2}
\end{align}
\end{subequations}

\subsection{Linearization via Aggregation of Products}
We now present an alternative linearization of the bilinear constraints~\eqref{model:1model_gate_bilinear} based on the observation that, due to constraints~\eqref{model:1model_assignnode2qubit}, the sum $\sum_{(i,j)\in \bar{E}} w^t_{pi} \cdot w^t_{qj}$ takes either the value zero or one for each $(p,q)\in C,\, t\in [T]$.
Thus, we can interpret it as a single binary variable.
Accordingly, we introduce binary \emph{connection} variables $z^t_{pq}\in\{0,1\}$ for all $(p, q)\in C,\, t\in [T]$ such that $z^t_{pq}=\sum_{(i,j)\in \bar{E}} w^t_{pi} \cdot w^t_{qj}$ holds.
Constraints~\eqref{model:1model_gate_bilinear} are then replaced by $\sum_{t\in [T]} z^t_{pq}\geq 1$ for all connections $(p,q)$.
The following model is obtained by linearizing the bilinear constraints $z^t_{pq}=\sum_{(i,j)\in \bar{E}} w^t_{pi} \cdot w^t_{qj}$.

\begin{subequations}
    \begin{align}
    \hypertarget{model:BPZ=}{\text{ZEQ}}: \min\; &\frac{1}{2}\sum_{t\in [T-1]}\sum_{p\in Q}\sum_{(i, j)\in \bar{E}}x^t_{pij}\\
    &\eqref{model:1model_assignqubit2node}-\eqref{model:1model_xcycle}, \eqref{model:1model_wvar}-\eqref{model:1model_xvar},&&\\
    &\sum_{t\in [T]} z^t_{pq}\geq 1,&& (p,q)\in C,
    \label{model:2amodel_gate_z}\\
    &\sum_{j\in V: N(i)\subseteq N(j)}w^t_{pj}+\sum_{j\in N(i)}w^t_{qj}\leq 1 + z^t_{pq}, && (p, q)\in C, i\in V, t\in [T],\label{model:2model_z_1block1_1}\\
    &\sum_{j\in V: N(i)\subseteq N(j)}w^t_{qj}+\sum_{j\in N(i)}w^t_{pj}\leq 1 + z^t_{pq}, && (p, q)\in C, i\in V, t\in [T],\label{model:2model_z_1block1_2}\\
    &\sum_{j\in V: V\setminus N(i)\subseteq V\setminus N(j)}w^t_{pj}-\sum_{j\in N(i)}w^t_{qj}\leq 1 - z^t_{pq}, && (p, q)\in C, i\in V, t\in [T],\label{model:2model_z_1block0_1}\\
    &\sum_{j\in V: V\setminus N(i)\subseteq V\setminus N(j)}w^t_{qj}-\sum_{j\in N(i)}w^t_{pj}\leq 1 - z^t_{pq}, && (p, q)\in C, i\in V, t\in [T],\label{model:2model_z_1block0_2}\\
    &z^t_{pq}\in \{0,1\}, && (p,q)\in C,\, t\in [T].
    \label{model:2amodel_zvar}
    \end{align}
\end{subequations}
Here, we linearized the constraint $z^t_{pq}=\sum_{(i,j)\in \bar{E}} w^t_{pi} \cdot w^t_{qj}$ as follows:
Due to constraints~\eqref{model:1model_assignnode2qubit}, for each connection $(p,q)$ and $t\in [T]$,
we have $w^t_{pi}=w^t_{qj}=1$ for two nodes $i,j\in V, \;i\neq j$.
If $(i,j)\in \bar{E}$, then $z^t_{pq}=1$ must hold, which can be modeled via $w^t_{pi}+w^t_{qj}\leq 1 + z^t_{pq}$, and if $(i,j)\notin \bar{E}$, then $z^t_{pq}=0$ must hold, which can be modeled via $w^t_{pi}+w^t_{qj}\leq 2- z^t_{pq}$.
Constraints~$\eqref{model:2model_z_1block1_1}-\eqref{model:2model_z_1block0_2}$ are a lifted version of these constraints.
More details on the lifting procedure can be found in \Cref{sec:few_constraints}.

As before, it suffices to model the inequality $z^t_{pq}\leq \sum_{(i,j)\in \bar{E}} w^t_{pi} \cdot w^t_{qj}$ instead of modeling equality.
This allows us to drop two of the four constraint sets~$\eqref{model:2model_z_1block1_1}-\eqref{model:2model_z_1block0_2}$ from model~\hyperlink{model:BPZ=}{\text{ZEQ}}, resulting in a model with fewer constraints:

\begin{subequations}
    \begin{align}
    \hypertarget{model:BPZ<}{\text{ZIEQ}}: \min \;&\frac{1}{2}\sum_{t\in [T-1]}\sum_{p\in Q}\sum_{(i, j)\in \bar{E}}x^t_{pij}\\
    &\eqref{model:1model_assignqubit2node}-\eqref{model:1model_xcycle}, \eqref{model:1model_wvar}-\eqref{model:1model_xvar}, \eqref{model:2amodel_gate_z},\eqref{model:2model_z_1block0_1},\eqref{model:2model_z_1block0_2},\eqref{model:2amodel_zvar}.&&\label{ZB:2}
    \end{align}
\end{subequations}

\subsection{Linear Descriptions of Sub-Polytopes}
\label{sec:polytopes}
Tight linear descriptions are crucial for a successful application of branch-and-cut algorithms since they result in strong bounds and thus help to avoid often branching.
In \Cref{app:sec_Ylinearization}, we show that the linearization used in model~\hyperlink{model:RLTY}{\text{YIEQ}} is tight in the sense that the constraints \eqref{model:1model_assignnode2qubit}, \eqref{model:1amodel_gate_y}, \eqref{model:1cmodel_rlt1}, and \eqref{model:1cmodel_rlt2} provide a linear description of the convex hull of all binary-valued points satisfying the constraints \eqref{model:1model_assignnode2qubit}, \eqref{model:1amodel_gate_y} and $y^t_{pqij}\leq w^t_{pi}\cdot w^t_{qj}$ for a fixed connection $(p, q)$.
More formally, we consider the polytope
\begin{align*}
    P^\leq_{pq}\define \conv\{(w^t_{pi}, w^t_{qj}, y^t_{pqij})\in \{0,1\}^{T(2|V|+|\bar{E}|)}:\;& \eqref{model:1model_assignnode2qubit}, \eqref{model:1amodel_gate_y}-\eqref{model:1amodel_mccormick2}
    \}.
\end{align*}
In \Cref{thm:convexhull_PXY_ieq_expanded}, we derive a linear description for a generalized family of polytopes.
For the special case of $P^\leq_{pq}$, we have
\begin{corollary}\label{convexhull_Ylinearization}
\[
    P^\leq_{pq} = \{(w^t_{pi}, w^t_{qj}, y^t_{pqij})\in [0,1]^{T(2|V|+|\bar{E}|)}:\;\eqref{model:1model_assignnode2qubit}, \eqref{model:1amodel_gate_y},\eqref{model:1cmodel_rlt1},\eqref{model:1cmodel_rlt2}
    \}.
\]
\end{corollary}
  
In \Cref{app:sec_Zlinearization}, we derive a linear description of the convex hull of all binary-valued points satisfying constraints \eqref{model:1model_assignnode2qubit} and either $z^t_{pq}=\sum_{(i,j)\in \bar{E}} w^t_{pi} \cdot w^t_{qj}$ or $z^t_{pq}\leq \sum_{(i,j)\in \bar{E}} w^t_{pi} \cdot w^t_{qj}$ for a fixed connection $(p, q)$ and $t\in [T]$.
To this end, we define 
\begin{align*}
    P_{pqt}&\define \conv\{(w^t_{pi},w^t_{qj},z^t_{pq})\in \{0,1\}^{2|V|+1}:\eqref{model:1model_assignnode2qubit},\; z^t_{pq}=\sum_{(i,j)\in \bar{E}}w^t_{pi}\cdot w^t_{qj}\},\\
    P^\leq_{pqt}&\define \conv\{(w^t_{pi},w^t_{qj},z^t_{pq})\in \{0,1\}^{2|V|+1}:\eqref{model:1model_assignnode2qubit},\; z^t_{pq}\leq \sum_{(i,j)\in \bar{E}}w^t_{pi}\cdot w^t_{qj}\}.
\end{align*}
Also in this case, we derive a linear description for a generalized family of polytopes in \Cref{thm:CXY=,thm:CXY<=}.
Applying these theorems to $P_{pqt}$ and $P^\leq_{pqt}$ yields
\begin{corollary}\label{convexhull_Zlinearization}
    Let $V'$ be a copy of $V$, and define the bipartite graph $G_{pqt}\coloneqq (V\,\dot{\cup}\, V',\{(i,j)\in V\times V': \{i,j\}\in E\})$.
    Then
    \begin{align*}
    &P_{pqt}
    =
    \conv\{(w^t_{pi},w^t_{qj},z^t_{pq})\in [0,1]^{2|V|+1}: \eqref{model:1model_assignnode2qubit},\eqref{con:ifthenrealmodeleq},\eqref{con:ifthenrealmodelieq0},\eqref{con:ifthenrealmodelieq1}\},\\
    &P_{pqt}^\leq 
    =
    \conv\{(w^t_{pi},w^t_{qj},z^t_{pq})\in [0,1]^{2|V|+1}: \eqref{model:1model_assignnode2qubit},\eqref{con:ifthenrealmodeleq},\eqref{con:ifthenrealmodelieq1}\},
    \end{align*}
    where the constraints are
    \begin{align}
        &\sum_{i\in I}w^t_{pi}+\sum_{j\in J}w^t_{qj}= 1 + z^t_{pq},\quad && (I, J)\in B(G_{pqt}): E(G_{pqt})=I\times J,
        \label{con:ifthenrealmodeleq}\\
        &\sum_{i\in I}w^t_{pi}+\sum_{j\in J}w^t_{qj}\leq 1 + z^t_{pq},\quad && (I, J)\in B(G_{pqt}): E(G_{pqt})\neq I\times J,
        \label{con:ifthenrealmodelieq0}\\
        &\sum_{i\in I}w^t_{pi}+\sum_{j\in J}w^t_{qj}\leq 2 - z^t_{pq},\quad && (I, J)\in AB(G_{pqt}): E(G_{pqt})\neq V\setminus V'\times Y\setminus J.
        \label{con:ifthenrealmodelieq1}
    \end{align}
    Here, $B(G_{pqt})$ and $AB(G_{pqt})$ denote the inclusion-maximal bicliques and antibicliques, respectively.
    For either of the two sets, an element is a pair $(I,J)$, where $I$ and $J$ denote the $V$ and $V'$ component of the (anti-)biclique $I\times J$.
\end{corollary}
The number of antibicliques in $G_{pqt}$ can grow very quickly in the number of nodes.
Computational experiments indicate an exponential growth of $|AB(G_{pqt})|$ with respect to the number of nodes $n$ for the path graph $P_n$ and the cycle graph $C_n$.
Hence, including the constraints~$\eqref{con:ifthenrealmodeleq}-\eqref{con:ifthenrealmodelieq1}$ would lead to prohibitively large integer programming models, which is why we instead use a smaller set of tight inequalities, $\eqref{model:2model_z_1block1_1}-\eqref{model:2model_z_1block0_2}$ for $P_{pqt}$ and \eqref{model:2model_z_1block0_1},\eqref{model:2model_z_1block0_2} for $P_{pqt}^\leq$.

\subsection{Symmetry Breaking Constraints}\label{sec:symmetries}
An integer programming model is said to be \emph{symmetric} if there exist permutations of its variables under which the model remains invariant.
High symmetry often hinders standard branch-and-cut algorithms for solving integer programs due to repeated exploration of isomorphic subproblems in the enumeration tree.
We say that a graph is \emph{symmetric} if it has a nontrivial automorphism, that is, an isomorphism onto itself.
From \Cref{cor:isomorphism_trafo} follows that our integer programming models are symmetric if $H$ or $A$ are symmetric.
In the following, we introduce constraints which break the symmetry in our models if $H$ or $A$ are symmetric.

\textbf{$\boldsymbol{H}$ is symmetric.}
Two nodes $i$ and $j$ are called \emph{similar} if there is an automorphism that maps $i$ to $j$.
Let $K$ be an inclusion-maximal set of non-similar nodes.
Then, there exists an optimal solution with $f_k(p)\in K$ where $k\in [T]$ and $p\in Q$ are arbitrary but fixed.
To eliminate symmetric solutions, we may include the following constraints for an arbitrary token $p\in Q$ and step $k=\max\{1, \lfloor T / 2\rfloor\}$:
\begin{equation}\label{hardware symmetry constraints}
    \sum_{i\in K}w^k_{pi}=1,\qquad w^t_{pj}=0,\quad j\in V,\,t\in [T]: \min_{i\in K}\text{dist}(i,j)\geq 1 + |k-t|.
\end{equation}

\textbf{$\boldsymbol{A}$ is symmetric.}
If $A=K_n$, then for any initial assignment $f:Q \to V$ and $t\in [T]$, there exists an optimal solution such that $f$ is the bijection at $t$.
Let $k=\max\{1, \lfloor T / 2\rfloor\}$.
Then, we can include the constraints
\begin{equation}\label{algorithm symmetry constraints}
    w^k_{pf(p)}=1,\quad  p \in Q,\qquad w^t_{pj}=0,\quad j\in V,\,t\in [T]: \text{dist}(f(p),j)\geq 1 + |k-t|.
\end{equation}

\section{Computational Experiments and Applications}
\label{sec:computational}
This section is divided into three parts.
In \Cref{sec:IPperformance}, we compare the integer programming models for the Token Meeting Problem introduced in \Cref{sec:ip_models}.
\Cref{sec:costofoptimality} studies the computational cost of obtaining swap-optimal solutions.
In \Cref{sec:qubit routing case study}, we present a complete solution method for the qubit routing problem with commuting gates and compare it against two heuristics and a near-optimal method.
To evaluate solution quality and performance under realistic conditions, we consider two hardware graphs:
a $3\times3$ grid $H_1$ and a graph $H_2$ consisting of two 5-node cycles sharing an edge, totaling 8 nodes, as shown in \Cref{fig:computational_hardware_graphs}.
The graphs are chosen to resemble existing hardware graphs, specifically grid~\cite{harrigan2021quantum} and heavy-hex~\cite{chamberland2020topological} architectures.
All computational experiments were conducted on a server with Intel Xeon Gold 6326 processors, 128 GB RAM and 32 cores.
Gurobi 12.0.2 was used to solve the integer programs.
The code is publicly available \href{https://github.com/MoritzStargalla/optimized_qubit_routing_for_commuting_gates/}{here}.

\begin{figure}
    \centering
    \includegraphics[page=9, width=0.375\textwidth]{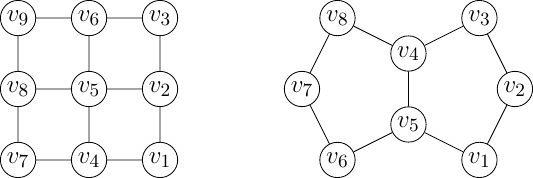}
    \caption{
        Graphs $H_1$ (left) and $H_2$ (right).
    }
    \label{fig:computational_hardware_graphs}
\end{figure}

\textbf{Data generation.}
For each hardware graph, we generate 20 algorithm graphs with varying densities $d\in\{0.05,0.1,\dots, 1\}$.
For each $d$, we construct a random graph with $\lceil n(n-1)/2 \cdot d \rceil$ edges, accepting the first connected graph found within 1000 attempts or the graph generated in the final attempt.
This resembles practical scenarios that typically involve connected algorithm graphs.

\textbf{Computing swap-optimal solutions.}
In principle, \Cref{alg:step_increase} can be employed to compute a swap-optimal solution.
However, we use a slightly improved version of \Cref{alg:step_increase}, which avoids solving multiple models with fewer steps than $\ML(H,A)$.
First, we compute $\ML(H,A, t)$ for increasing $t\in \N$ until $\ML(H,A, t)<\infty$.
Then $t=\MT(H,A)$ holds.
In \Cref{sec:IPperformance}, we compare the runtimes of different integer programming models for this procedure.
Finally, we compute $\ML(H,A, t^*)$ for $t^*=\ML(H,A, \MT(H,A))-1$.
By \Cref{prop:termination_increasing_steps}, this either yields a swap-optimal solution or proves that $\ML(H,A,\MT(H,A))=\ML(H,A)$ holds if the integer program is infeasible.
This second step is examined in \Cref{sec:costofoptimality}.

\subsection{Comparison of Different Models}\label{sec:IPperformance}
We compare the models \hyperlink{model:MCY}{\text{YEQ}}, \hyperlink{model:RLTY}{\text{YIEQ}}, \hyperlink{model:BPZ=}{\text{ZEQ}}, and \hyperlink{model:BPZ<}{\text{ZIEQ}} as introduced in \Cref{sec:ip_models}.

\textbf{Solution approach.} Each instance is solved by iteratively increasing the number of steps, starting from the lower bound in \Cref{lem:lower_step_bound} until a feasible solution is found.
A time limit of 7200 s is applied per run.
The runtime of \hyperlink{model:BPZ<}{\text{ZIEQ}} is used as a benchmark $T^{\rm bench}$ and the runtimes $T$ of other models are scaled as $T^{\rm scaled}=T/T^{\rm bench}$.
The values of $T^{\rm bench}$ and $T^{\rm scaled}$ are shown in \Cref{fig:eval_h2_scaled_runtimes}
while step and swap counts are shown in \Cref{fig:eval_swaps_depth}.

\begin{figure}
    \centering
    \includegraphics[width=0.8\textwidth]{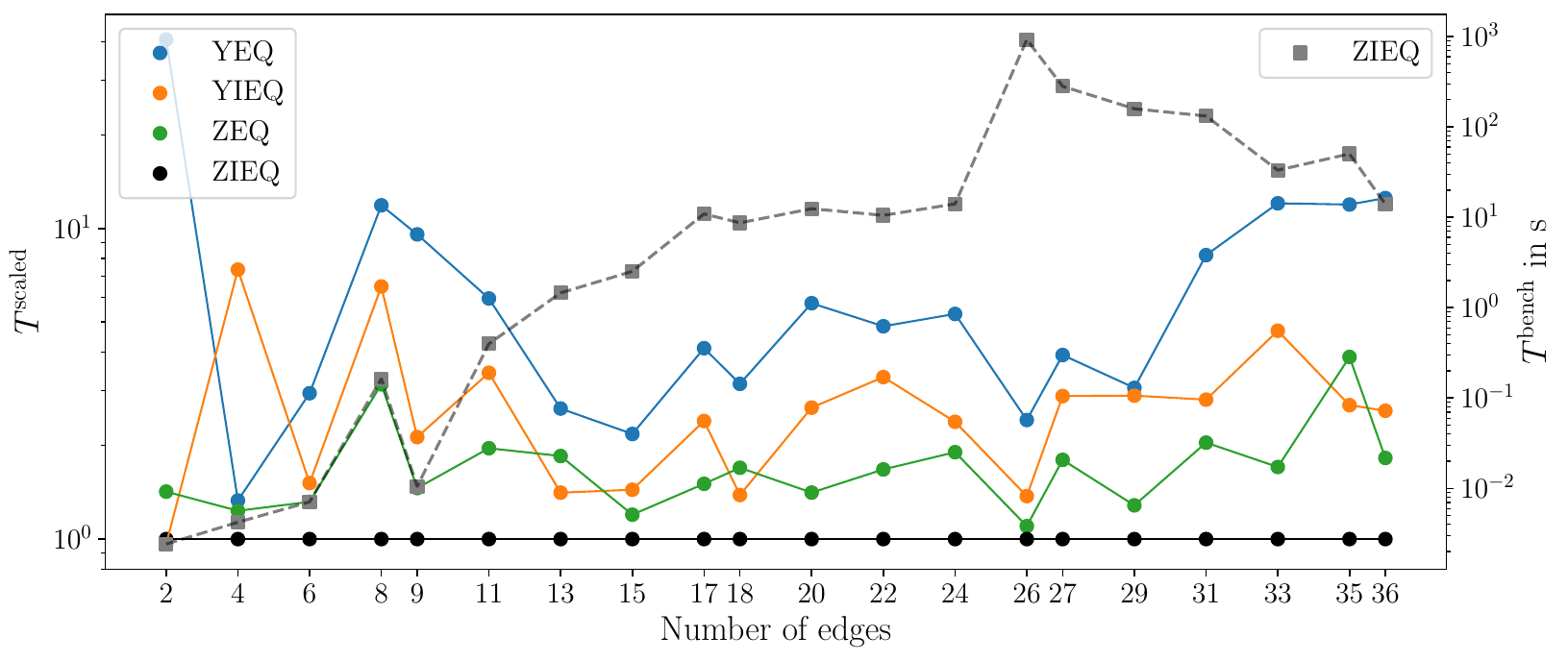}
    \includegraphics[width=0.8\textwidth]{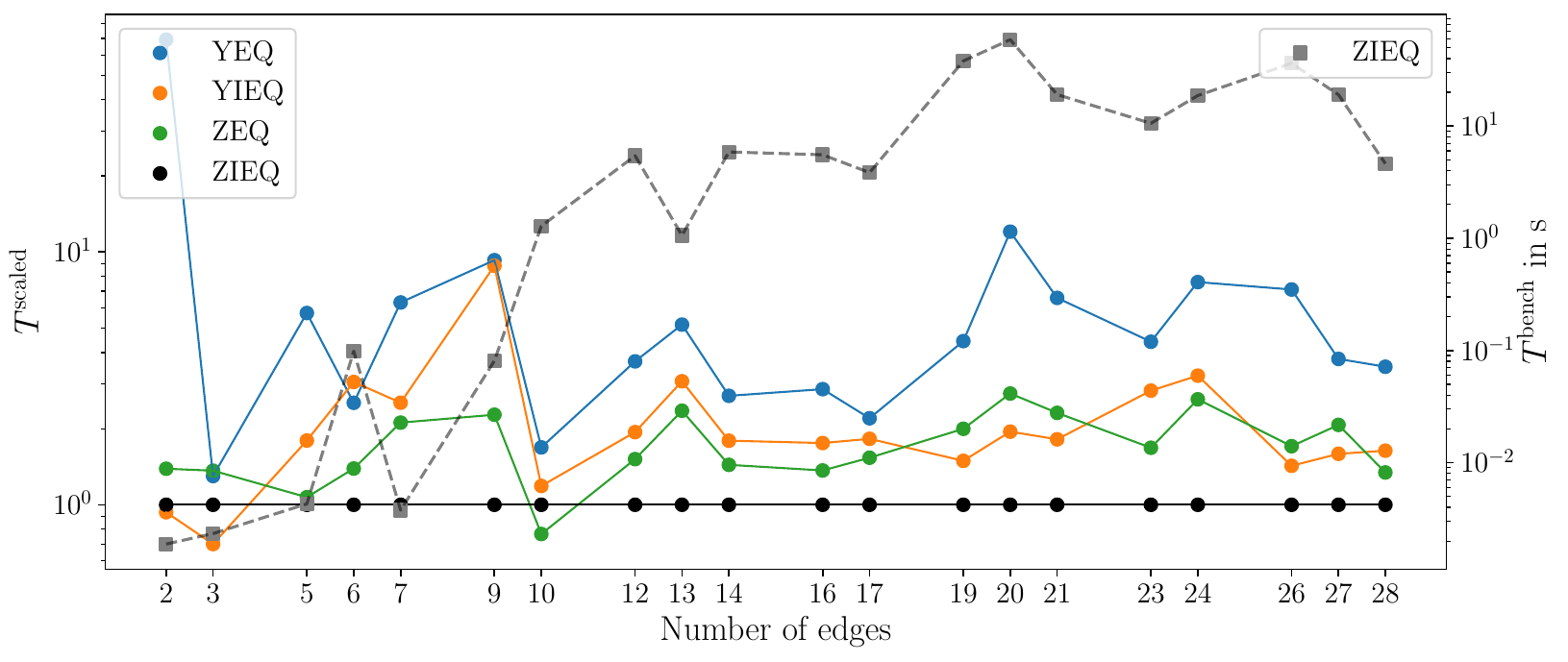}
    \caption{
        Runtime comparison for all models and instances of $H_1$ (top) and $H_2$ (bottom).
        The x-axis represents the number of edges in each instance.
        For each configuration, the values $T^{\rm scaled}$ (solid lines) are plotted with respect to the left y-axis, while  $T^{\rm bench}$ (dashed, gray) is plotted with respect to the right y-axis.
        Both y-axes are scaled logarithmically.
    }
    \label{fig:eval_h2_scaled_runtimes}
\end{figure}

\begin{figure}
    \centering
    \subfloat{
        \includegraphics[width=0.425\textwidth]{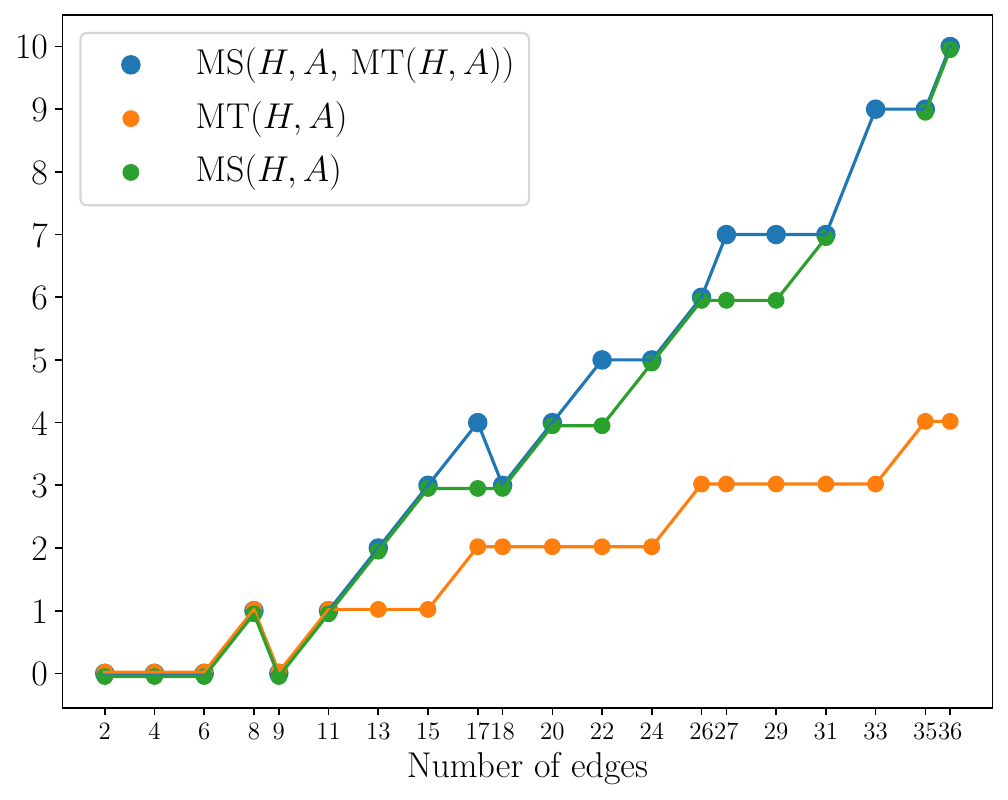}
    }
    \hspace{0.075\textwidth}
    \subfloat{
        \includegraphics[width=0.425\textwidth]{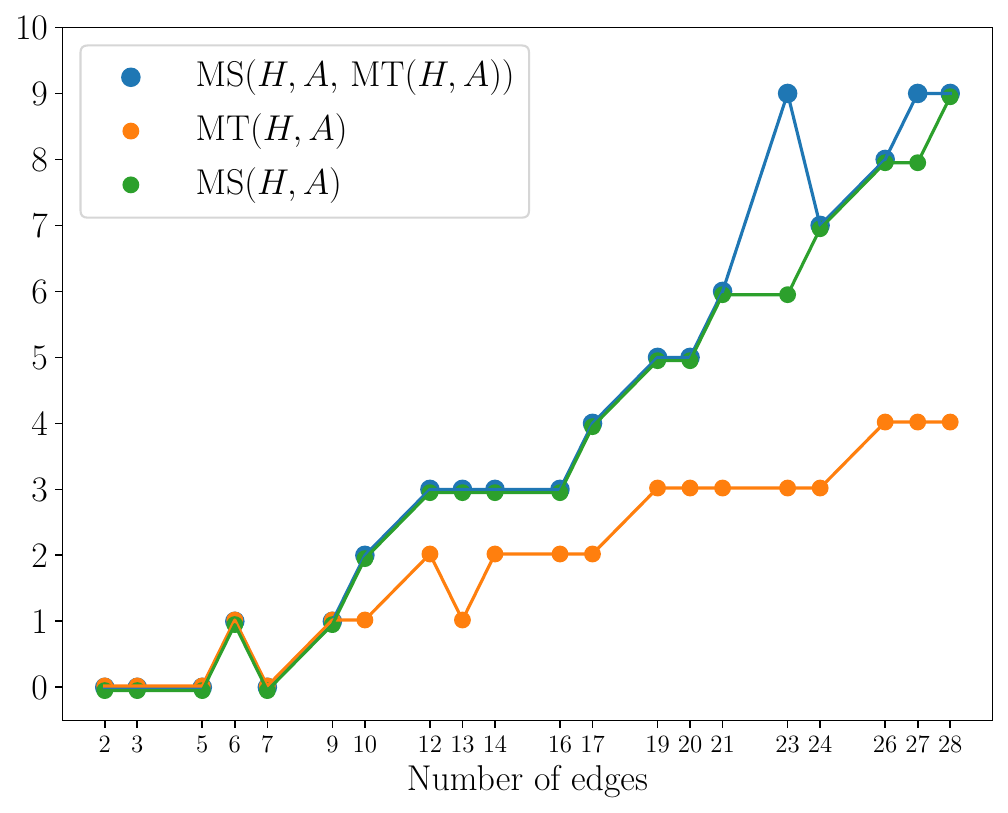}
    }
    \caption{
        Minimum number of steps (blue) and the minimum number of swaps (green) for instances of $H_1$ (left) and $H_2$ (right), as well as $\ML(H,A,\MT(H,A))$ (blue).
    }
    \label{fig:eval_swaps_depth}
\end{figure}

\textbf{Results on $H_1$.}
\hyperlink{model:MCY}{\text{YEQ}} incurred the highest runtimes for nearly all instances.
Excluding the outlier instance with 2 edges, \hyperlink{model:RLTY}{\text{YIEQ}} achieved a speed-up of $2.2\times$ on average compared to \hyperlink{model:MCY}{\text{YEQ}}; \hyperlink{model:BPZ=}{\text{ZEQ}} and \hyperlink{model:BPZ<}{\text{ZIEQ}} have average speed-up factors of $3.3$ and $6.0$ compared to \hyperlink{model:MCY}{\text{YEQ}}, respectively.

\textbf{Results on $H_2$.}
Again, \hyperlink{model:MCY}{\text{YEQ}} had the highest runtimes.
Excluding the outlier instance with 2 edges, \hyperlink{model:RLTY}{\text{YIEQ}} offers a $2.4\times$ average speed-up compared to \hyperlink{model:MCY}{\text{YEQ}}.
\hyperlink{model:BPZ=}{\text{ZEQ}} and \hyperlink{model:BPZ<}{\text{ZIEQ}} achieved average speed-up factors of $2.7$ and $4.9$ compared to \hyperlink{model:MCY}{\text{YEQ}}, respectively.

We conclude that \hyperlink{model:BPZ<}{\text{ZIEQ}} consistently outperforms all other models, including \hyperlink{model:BPZ=}{\text{ZEQ}}, suggesting that modeling the inequality $z^t_{pq}\leq \sum_{(i,j)\in \bar{E}}w^t_{pi}\cdot w^t_{qj}$ performs better than modeling the equality $z^t_{pq}= \sum_{(i,j)\in \bar{E}}w^t_{pi}\cdot w^t_{qj}$.
Maximum runtimes for \hyperlink{model:BPZ<}{\text{ZIEQ}} were 926 s (26 edges, $H_1$) and 59 s (20 edges, $H_2$).

\textbf{Symmetry breaking using the hardware graph.}
Given the symmetry of $H_1$ and $H_2$, we applied the the symmetry-breaking constraints~\eqref{hardware symmetry constraints} as described in \Cref{sec:symmetries}.
However, including these constraints resulted in a performance degrade in our experiments.

\textbf{Symmetry breaking using the algorithm graph.}
For the case $A=K_n$, i.e., density $d=1$, we compare runtimes with and without fixing the initial bijection via constraints~\eqref{algorithm symmetry constraints} as described in \Cref{sec:symmetries}.
The runtimes are reported in \Cref{tab:eval_completegraph_fixing}.
\begin{table}
    \centering
    {\small
        \addtolength{\tabcolsep}{-0.4em}
    \begin{tabular}{lcccccccc|}
        $H_i$-Strategy & \hyperlink{model:MCY}{\text{YEQ}}& \hyperlink{model:RLTY}{\text{YIEQ}}& \hyperlink{model:BPZ=}{\text{ZEQ}}& \hyperlink{model:BPZ<}{\text{ZIEQ}} \\
        \hline
        $H_1$-Standard & 175.1 & 36.2 & 25.5 & 14.0 \\
        $H_1$-Fixing & 14.2 & 7.9 & 10.0 & 5.2 \\
        \hline
        $H_2$-Standard & 16.3 & 7.6 & 6.2 & 4.6 \\
        $H_2$-Fixing & 4.8 & 2.2 & 2.6 & 1.6 \\
        \hline
    \end{tabular}}
    \caption{The total runtime in seconds for solving the models with input $(H_1, K_9)$ and $(H_2, K_8)$ with fixing (Fixing) and without fixing (Standard) the initial placement.}
    \label{tab:eval_completegraph_fixing}
\end{table}
All models benefited from significant runtime reductions, with speed-up factors ranging from $2\times$ to $12\times$.
We therefore strongly recommend using this fixing strategy when solving instances with $A=K_n$.
This arises, for example, in QAOA applied to the Sherrington–Kirkpatrick model~\cite{harrigan2021quantum}.
Moreover, the resulting solutions for $(H, K_n)$ can be used in the method of Weidenfeller et al.~\cite{Weidenfeller_2022}.

\subsection{The Cost of Optimality}\label{sec:costofoptimality}
In \Cref{sec:IPperformance}, we have computed the value of $\ML(H,A,\MT(H,A))$.
To obtain a swap-optimal solution, we now solve the fastest model \hyperlink{model:BPZ<}{\text{ZIEQ}} with $\ML(H,A,\MT(H,A))-1$ many steps,
that is, we compute $\ML(H,A,\ML(H,A,\MT(H,A))-1)$.
If model \hyperlink{model:BPZ<}{\text{ZIEQ}} is infeasible, then $\ML(H,A,\MT(HA))=\ML(H, A)$ holds; if it is feasible, we obtain a swap-optimal solution by \Cref{prop:termination_increasing_steps}.
To reduce the symmetry of the final model, we slightly modify model \hyperlink{model:BPZ<}{\text{ZIEQ}}.
\begin{subequations}
    \begin{align}
    \hypertarget{model:ZBOPT}{\text{ZOPT}}: \min \quad&\sum_{t\in [T-1]}s_t\\
    &\eqref{ZB:2},&&\\
    &\sum_{p\in Q}\sum_{(i,j)\in \bar{E}}x^t_{pij}=2s_t,&&t\in [T-1],\label{con:oneswap}\\
    &s_{t+1}\leq s_t,&&t\in [T-2],\label{con:symmetry1}\\
    &z^{t+1}_{pq}\leq s_t,&&(p,q)\in C,\,t\in [T-1],\label{con:symmetry2}\\
    &s_t\in \{0,1\},&&t\in [T-1].
    \end{align}
\end{subequations}
In this model, the binary variable $s_t\in \{0,1\}$ indicates whether a swap is performed in the step from
$t$ to $t+1$ for $t\in [T-1]$. 
Constraints~\eqref{con:oneswap} enforce this logic.
Note that we can restrict ourselves to solutions with one swap per step since every swap-optimal solution with up to $\ML(H,A,\MT(H,A))-1$ swaps will be feasible for model \hyperlink{model:ZBOPT}{ZOPT}.
Constraints~\eqref{con:symmetry1} impose an ordering on the variables $s_t$, which breaks symmetry in instances with $\ML(H,A)<\ML(H,A,\MT(HA))-1$ by enforcing that steps without swaps appear at the end.
Similarly, constraints~\eqref{con:symmetry2} further break symmetry by requiring that the connection indicator variables $z^{t}_{pq}$ can only be set to 1 in step $t$ if the preceding step contains a swap.

\Cref{fig:optimality_times} shows the total runtime obtained by summing the time to compute $\ML(H,A,\MT(H,A))$ using model \hyperlink{model:BPZ<}{\text{ZIEQ}} and the time to compute $\ML(H,A)$ by additionally solving model \hyperlink{model:ZBOPT}{ZOPT}.
This combined approach is denoted as \hyperlink{model:BPZ<}{\text{ZIEQ}} + \hyperlink{model:ZBOPT}{ZOPT}.
Model \hyperlink{model:ZBOPT}{ZOPT} was solved within 24 hours for all instances, except for the instance with 33 edges for $H_1$.
We observe that computing $\ML(H,A)$ significantly increases the runtime compared to computing $\ML(H,A,\MT(H,A))$ - in some cases by a factor of $1,000$.
In model \hyperlink{model:ZBOPT}{ZOPT}, a solution with $\ML(H,A)$ swaps was typically found relatively fast.
Closing the optimality gap required the majority of the total runtime.
The number of computed values of $\ML(H,A)$ and $\ML(H,A,\MT(H,A))$ are shown in \Cref{fig:eval_swaps_depth}.
We observe that in most instances $\ML(H,A)=\ML(H,A,\MT(H,A))$ holds.
A difference of one swap, $\ML(H,A)=\ML(H,A,\MT(H,A))-1$, also occurs frequently. 
The largest observed difference of $3$ swaps appeared in the instance with 23 nodes on $H_2$.
We conclude that computing $\ML(H,A,\MT(H,A))$ consistently yields close-to-optimal solutions while requiring significantly less computation time.

\begin{figure}
    \centering
    \subfloat{
        \includegraphics[width=0.425\textwidth]{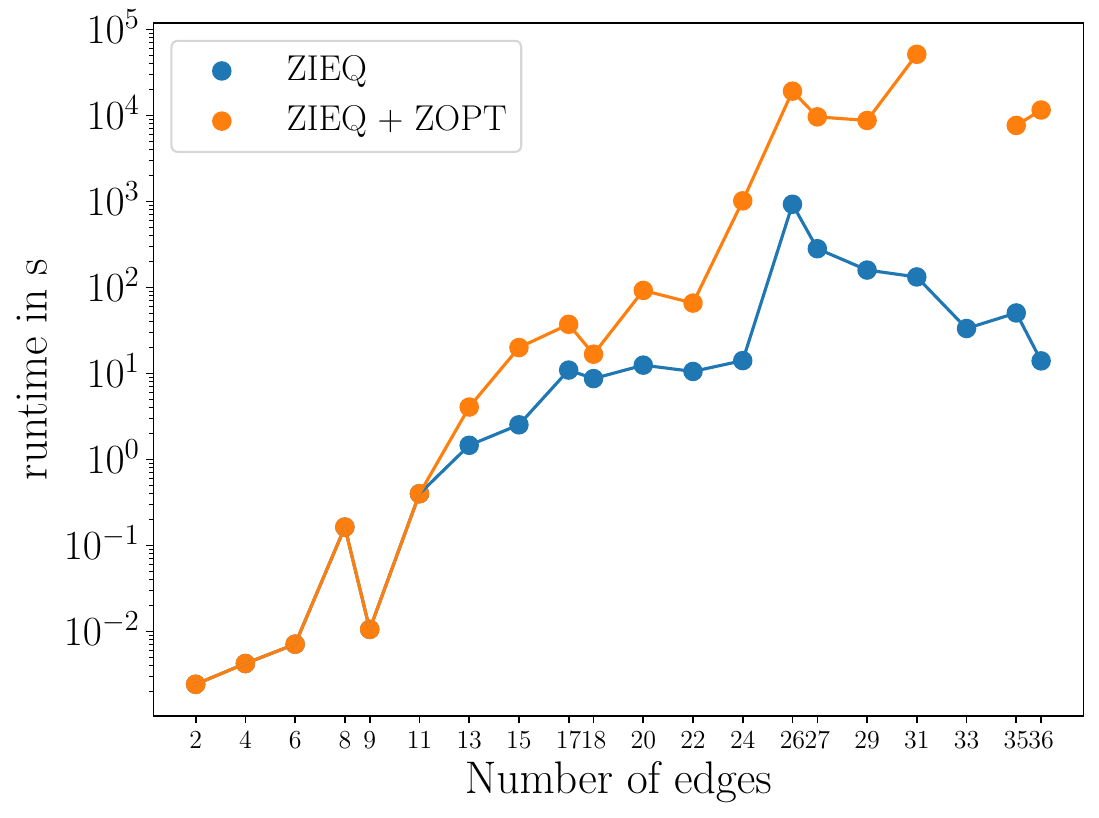}
    }
    \hspace{0.075\textwidth}
    \subfloat{
        \includegraphics[width=0.425\textwidth]{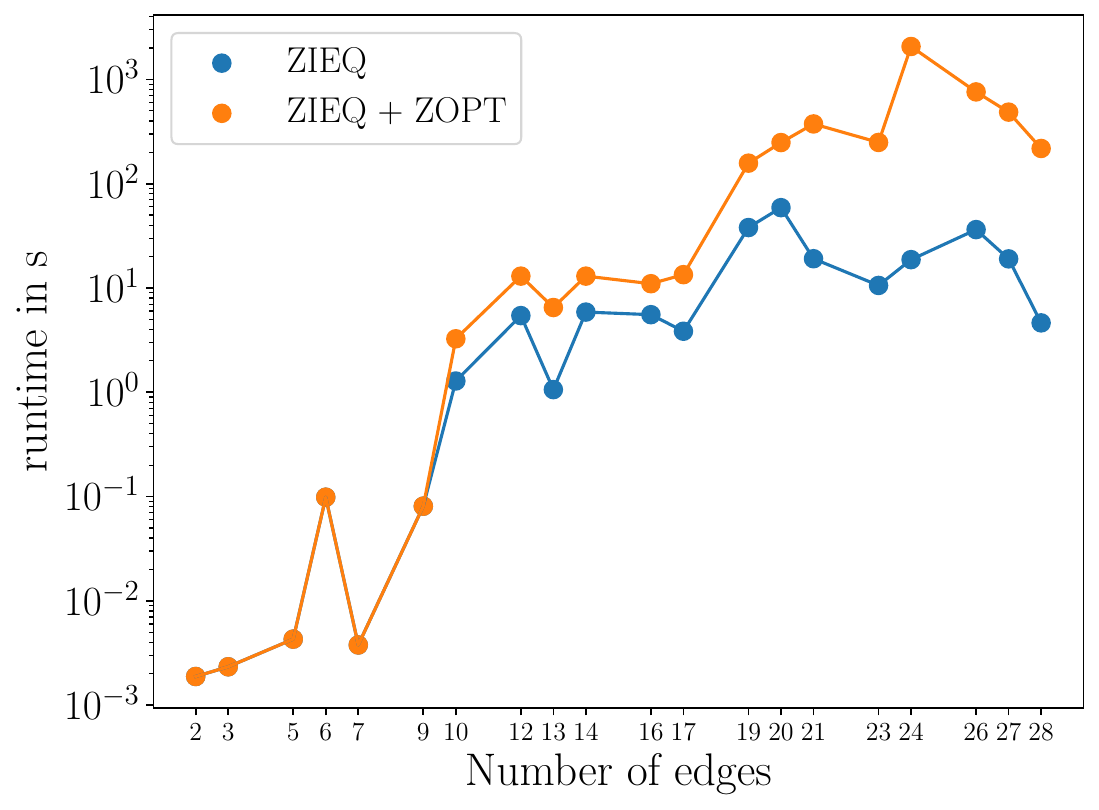}
    }
    \caption{
        Runtime comparison between \protect\hyperlink{model:BPZ<}{\text{ZIEQ}} and \protect\hyperlink{model:BPZ<}{\text{ZIEQ}} + \protect\hyperlink{model:ZBOPT}{ZOPT} for $H_1$ (left) and $H_2$ (right).
    }
    \label{fig:optimality_times}
\end{figure}
    
\subsection{Qubit Routing with Commuting Gates}\label{sec:qubit routing case study}
Next, we compare our two-step approach to the qubit routing problem with commuting gates to existing methods from literature.
To obtain a routed quantum circuit, the swap solution derived from the TMP solution must be converted into a full solution.
Our procedure for this task is illustrated in \Cref{fig:scheduling_ip_illustration} and summarized as follows.
Assume that we are given a swap solution $(f_1, (M_1,\dots, M_T))$ with $T$ steps.
To obtain a quantum circuit, we insert the swap gates corresponding to each matching $M_t$ into a single swap layer $S_t$.
Next, we create a number of empty layers after each swap layer and before the first one.
Now, the task is to assign the gates to either the swap layers or the empty layers such that, in each layer, the scheduled gates form a matching in the hardware graph.
For each gate, let $T^s_{pq}\subseteq [T]$ denote the set of steps $t$ such that the gate $(p,q)$ can be scheduled directly into the swap layer $S_t$, and let $T_{pq}\subseteq [T+1]$ denote the set of steps $t$ such that the gate can be inserted into an empty layer prior to the swap layer $S_t$ (for notational simplicity, we introduce an artificial, empty swap layer $S_{t+1}$).
To ensure that all feasible solutions included, we must add a sufficient number of empty layers.
For each $t\in [T+1]$, define the subgraph $H_t=(V, E_t)$ of the hardware graph $H$, where $\{i,j\}\in E_t$ if and only if $\{i,j\}\in E$ and $\{f^{-1}_t(i),f^{-1}_t(j)\}\in C$.
This represents the set of edges in the hardware graph that correspond to executable gates with respect to the bijection $f_t$.
Inserting gates into $k$ new layers naturally defines an edge coloring of $H_t$ with $k$ colors.
All edges corresponding to the gates in the same layer are assigned a unique color.
Since $H_t$ can be colored with at most $\Delta(H_t)+1$ many colors (by Vizing's Theorem), introducing $\Delta(H_t)+1$ empty layers prior to the swap layer $S_t$ suffices to ensure any depth-optimal gate insertion is feasible.

To obtain a depth-optimal gate insertion, we develop an integer programming model.
The model contains binary \emph{layer} variables $u_{tb}\in \{0,1\}$ for all $t\in [T+1],\,b\in [\Delta(H_t) + 1]$, where $u_{tb}=1$ indicates that the $b$-th layer prior to the swap layer $S_t$ is used.
Additionally, it contains binary \emph{scheduling} variables $a_{pqtb}\in \{0,1\}$ for all $(p,q)\in C,\, t\in T_{pq},\, b\in [\Delta(H_t) + 1]$, where $a_{pqtb}=1$ indicates that gate $(p,q)$ is applied in the $b$-th empty layer prior to the swap layer $S_t$.
For all $(p,q)\in C,\, t\in T^s_{pq}$, the model contains binary variables $a_{pqt0}\in\{0,1\}$, where $a_{pqt0}=1$ indicates that gate $(p,q)$ is scheduled in the swap layer $S_t$.
The integer program is formulated as follows:
\begin{figure}[t]
    \centering
    \includegraphics[page=13, width=0.7\textwidth]{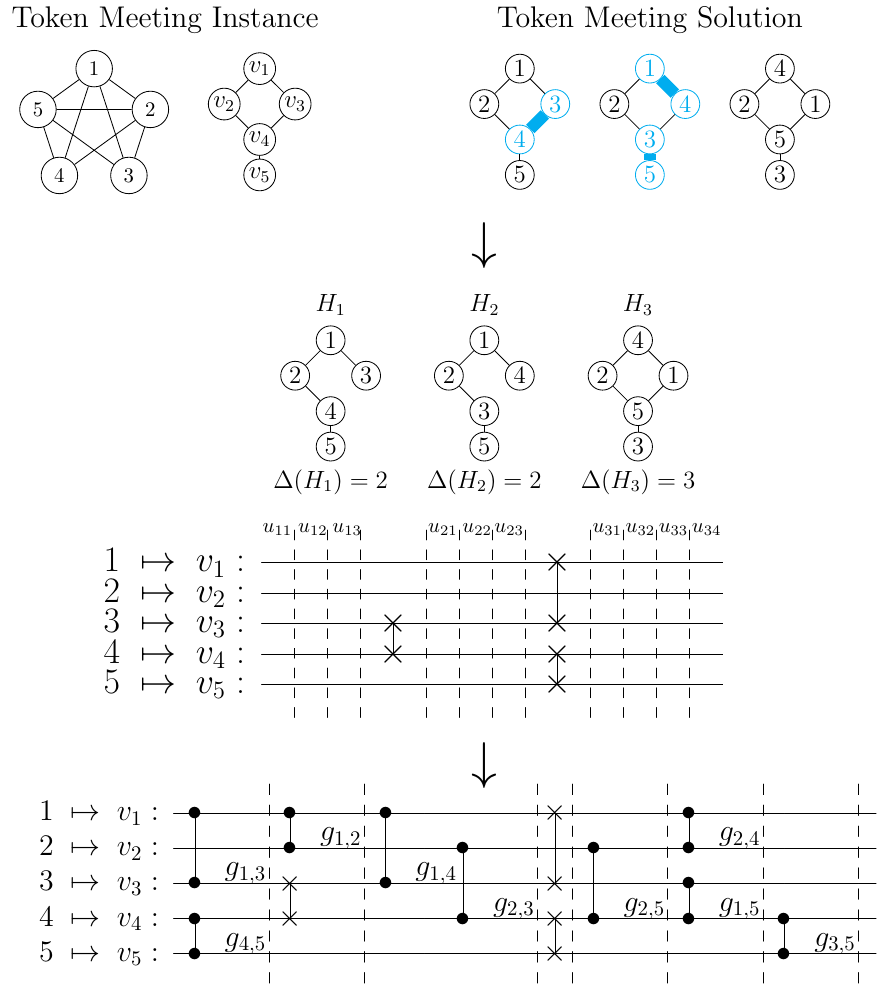}
    \caption{Visualization of the process of turning a TMP solution into a routed circuit.}
    \label{fig:scheduling_ip_illustration}
\end{figure}
\begin{subequations}\label{model:4model}
    \begin{align}
    \min &\sum_{t\in [T+1]}\sum_{b\in [\Delta(H_t)+1]}u_{tb}&&
    \label{model:4model_obj}\\
    &\sum_{t\in T^s_{pq}}a_{pqt0} + \sum_{t\in T_{pq}}\sum_{b\in [\Delta(H_t) + 1]}a_{pqtb} = 1, \quad&& (p, q)\in C,
    \label{model:4model_schedule_gate}\\
    &\sum_{q\in N(p):\,  t \in T_{p, q}}a_{pqtb} \leq u_{tb}, \quad&&  p\in Q,\,  t\in [T+1],\, b\in [\Delta(H_t)+1],
    \label{model:4model_newlayer_matching}\\
    &\sum_{q\in N(p):\, t \in T^s_{p, q}}a_{pqt0} \leq 1, \quad && p\in Q\setminus f^{-1}_t(M_t),\, t\in [T],
    \label{model:4model_swaplayer_matching}\\
    &u_{t(b + 1)}\leq u_{tb}, \quad && t\in [T],\, b\in [\Delta(H_t)],
    \label{model:4model_symmetry}\\
    &u_{tb}\in\{0,1\}, \quad&& t\in [T+1],\, b\in [\Delta(H_t)+1],
    \label{model:4model_uvar}\\
    &a_{pqtb}\in \{0,1\}, \quad&&  (p,q)\in G,\, t\in T_{pq},\,b\in [\Delta(H_t)+1],
    \label{model:4model_avar_new}\\
    &a_{pqt0}\in \{0,1\}, \quad&&  (p,q)\in G,\, t\in T^s_{pq}.
    \label{model:4model_avar_old}
\end{align}
\end{subequations}
The constraints \eqref{model:4model_schedule_gate} ensure that every gate is scheduled in exactly one layer.
Constraints \eqref{model:4model_newlayer_matching} and \eqref{model:4model_swaplayer_matching} ensure that no two gates sharing a node are scheduled in the same layer.
Furthermore, constraints \eqref{model:4model_newlayer_matching} link the scheduling and layer usage variables.
The constraints \eqref{model:4model_symmetry} impose an ordering on the layer usage variables, thereby breaking symmetry of the model.
The objective \eqref{model:4model_obj} aims to minimize the number of additionally introduced layers.
Finally, \eqref{model:4model_uvar} - \eqref{model:4model_avar_old} define the ranges of the variables.

Note that scheduling all swap gates belonging to the same matching in a single swap layer might exclude depth-optimal solutions as it can be beneficial to distribute these swap gates over multiple layers.
We use this modeling choice since it allows for a compact integer programming model.

\textbf{Comparison with heuristics.}
We compare our compilation approach against established heuristics in terms of runtime, number of swaps and the depth of the routed circuits.
The following methods are evaluated:

\textbf{(CQR)} A TMP solution is computed by solving model~\hyperlink{model:BPZ<}{\text{ZIEQ}}, followed by the construction of a routed quantum circuit using the scheduling integer program~\eqref{model:4model}.
For (CQR), the runtime (SCHE) of the scheduling integer program~\eqref{model:4model} runtime is also reported separately; the runtime of (CQR) is the sum of the runtimes of \hyperlink{model:BPZ<}{\text{ZIEQ}} and (SCHE).

\textbf{(SAB)} The circuit is transpiled using the SABRE compiler from Qiskit 2.1.0.~\cite{li2019tackling,qiskit2024}.
To ensure a fair comparison, \texttt{SabreLayout} is executed with \texttt{swap\_trials} and \texttt{layout\_trials} both set to 1000, allowing more iterations than the default settings.

\textbf{(WEI)} The circuit is transpiled using the \texttt{Commuting2qGateRouter} method in Qiskit, proposed by Weidenfeller et al.~\cite{Weidenfeller_2022}.
This method requires a user-supplied swap-layer strategy and an edge coloring of the hardware graph.
The performance is highly sensitive to these inputs.
To ensure a fair comparison, we supply a swap-layer strategy attaining $\ML(H, K_n, \MT(H, K_n))$ swaps together with a minimal edge coloring.
The corresponding swap strategies and edge colorings for $H_1$ and $H_2$ are shown in \Cref{fig:computational_hardware_graphs_Kn_solution,fig:computational_hardware_graphs_edge_coloring}.

\begin{figure}[t]
    \centering
    \begin{subfigure}[h]{0.8\linewidth}
        \centering
        \includegraphics[page=10, width=\textwidth]{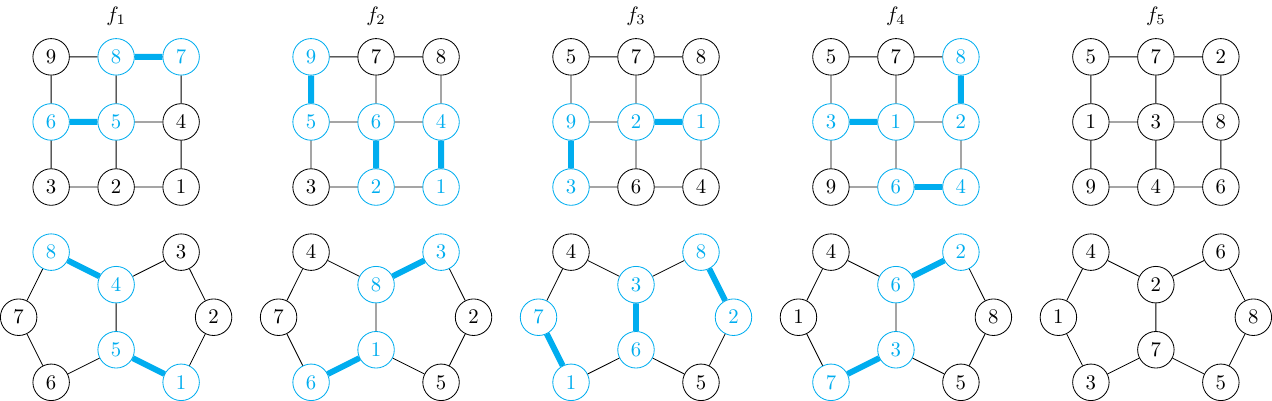}
    \caption{
        Solutions with $\ML(H, K_n, \MT(H, K_n))$ swaps and $\MT(H, K_n)$ steps for $(H_1, K_9)$ (top) and $(H_2, K_8)$ (bottom).
        For each $t\in [T]$, the placement $f_t$ is shown and the edges involved in a swap are highlighted in blue.
    }
    \label{fig:computational_hardware_graphs_Kn_solution}
    \end{subfigure}
    \hfill
    \begin{subfigure}[h]{0.15\linewidth}
        \centering
        \includegraphics[page=11, width=\textwidth]{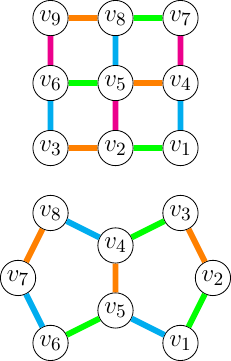}
        \caption{
            Edge colorings of $H_1$ and $H_2$.
    }
\label{fig:computational_hardware_graphs_edge_coloring}
\end{subfigure}
\label{fig:computational_hardware_graphs_sol_col}
\end{figure}

The results for runtime, number of swap gates, and depth are shown in \Cref{fig:eval_h1_swap_scaled,fig:eval_h2_swap_scaled,fig:eval_h1_depth_scaled,fig:eval_h2_depth_scaled,fig:eval_h1_time_logscale,fig:eval_h2_time_logscale}.

\begin{figure}
    \begin{subfigure}[h]{0.4\linewidth}
        \centering
        \includegraphics[width=\textwidth]{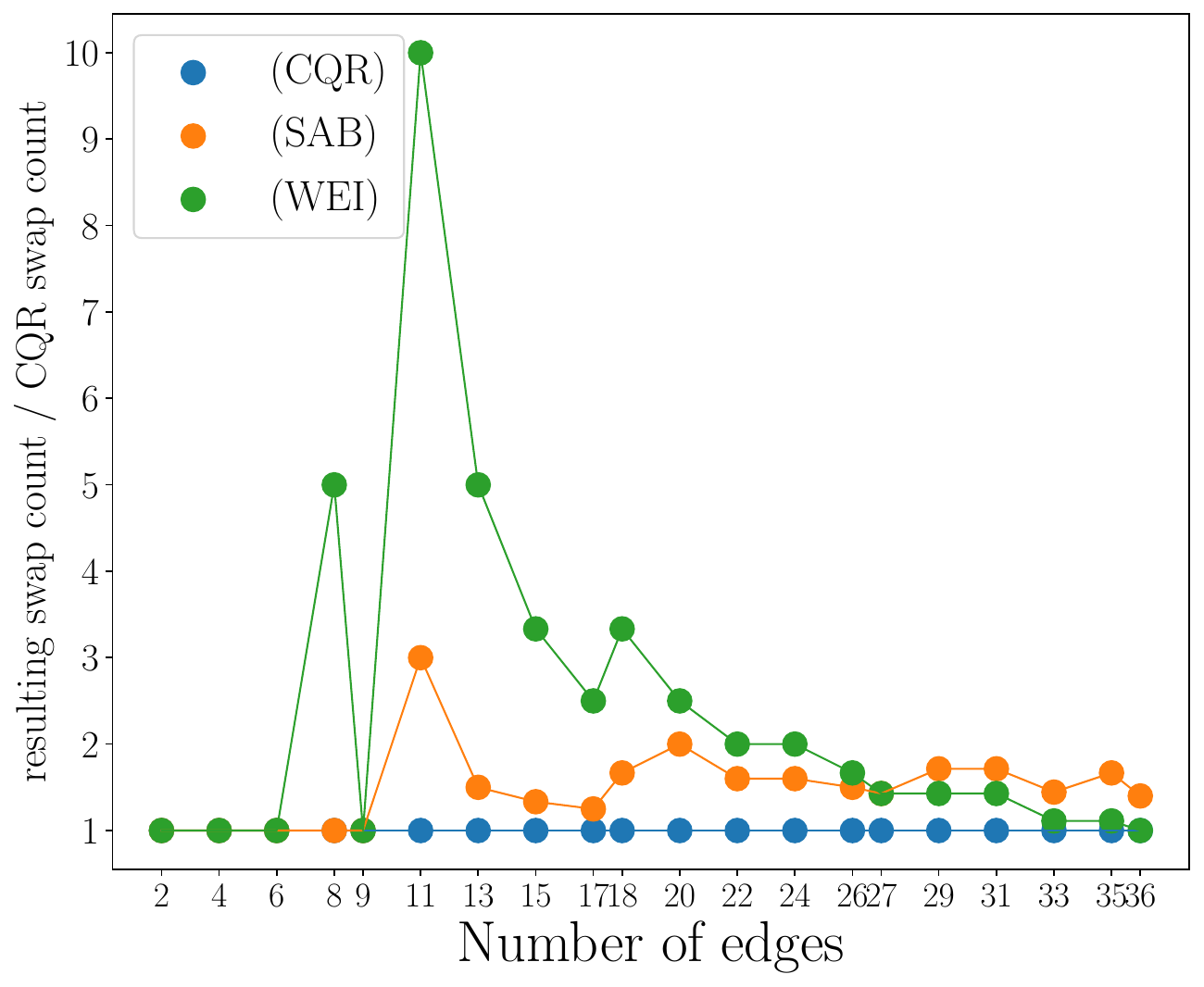}
        \caption{
            Swap comparison on $H_1$.
        }
        \label{fig:eval_h1_swap_scaled}
    \end{subfigure}
    \hfill
    \begin{subfigure}[h]{0.4\linewidth}
        \centering
        \includegraphics[width=\textwidth]{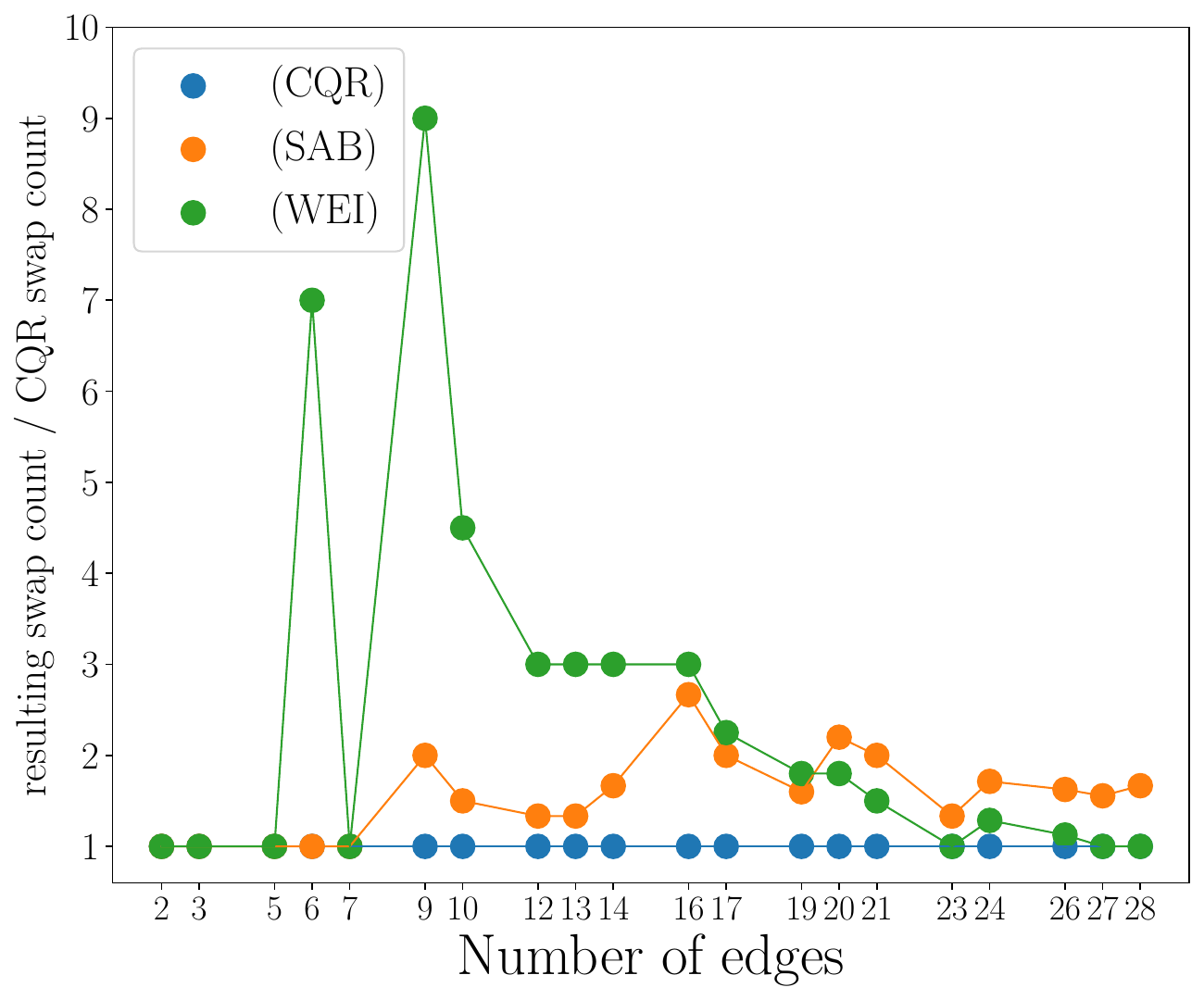}
        \caption{
            Swap comparison on $H_2$.
        }
        \label{fig:eval_h2_swap_scaled}
    \end{subfigure}
    \begin{subfigure}[h]{0.4\linewidth}
        \centering
        \includegraphics[width=\textwidth]{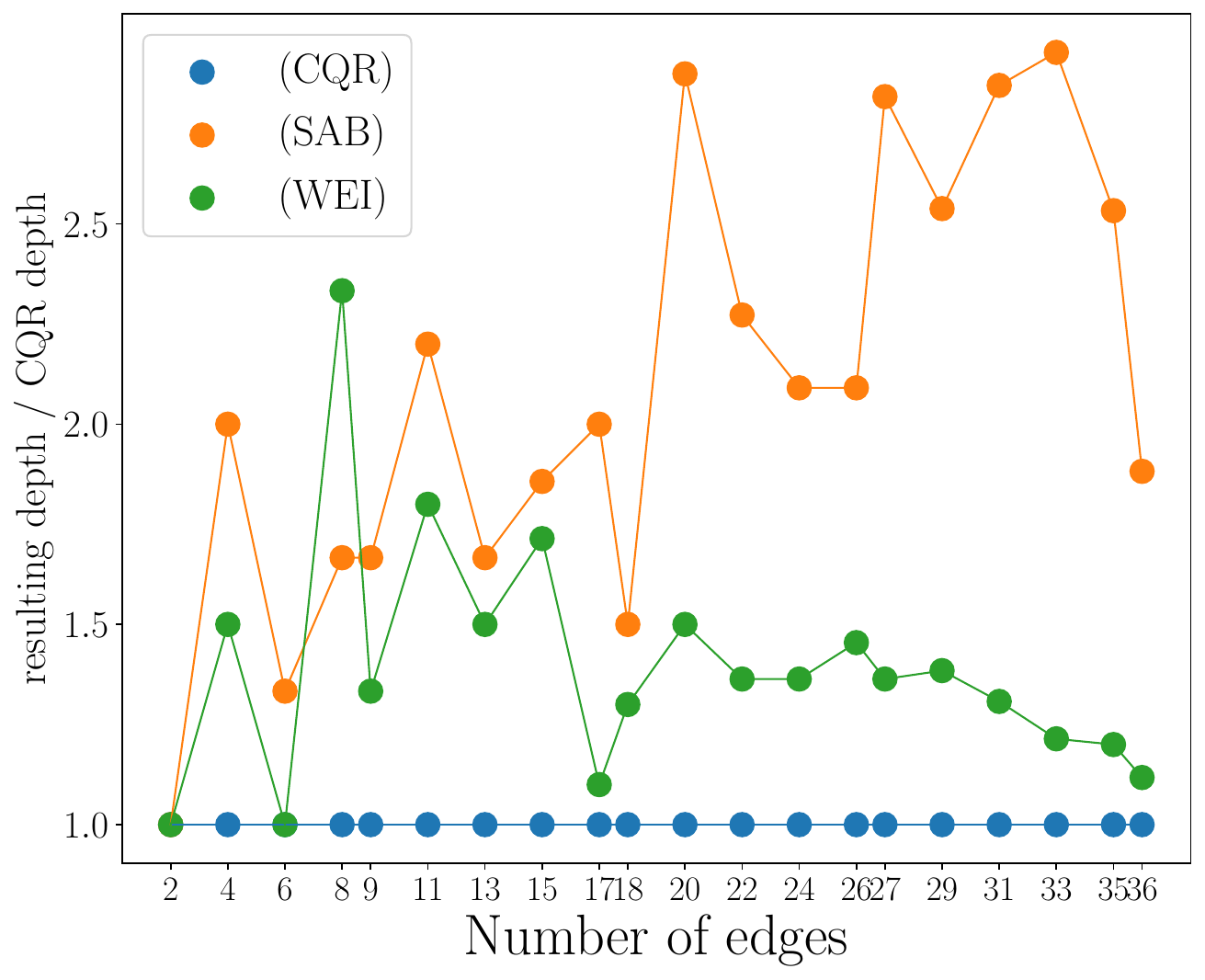}
        \caption{
            Depth comparison on $H_1$.
        }
        \label{fig:eval_h1_depth_scaled}
    \end{subfigure}
    \hfill
    \begin{subfigure}[h]{0.4\linewidth}
        \centering
        \includegraphics[width=\textwidth]{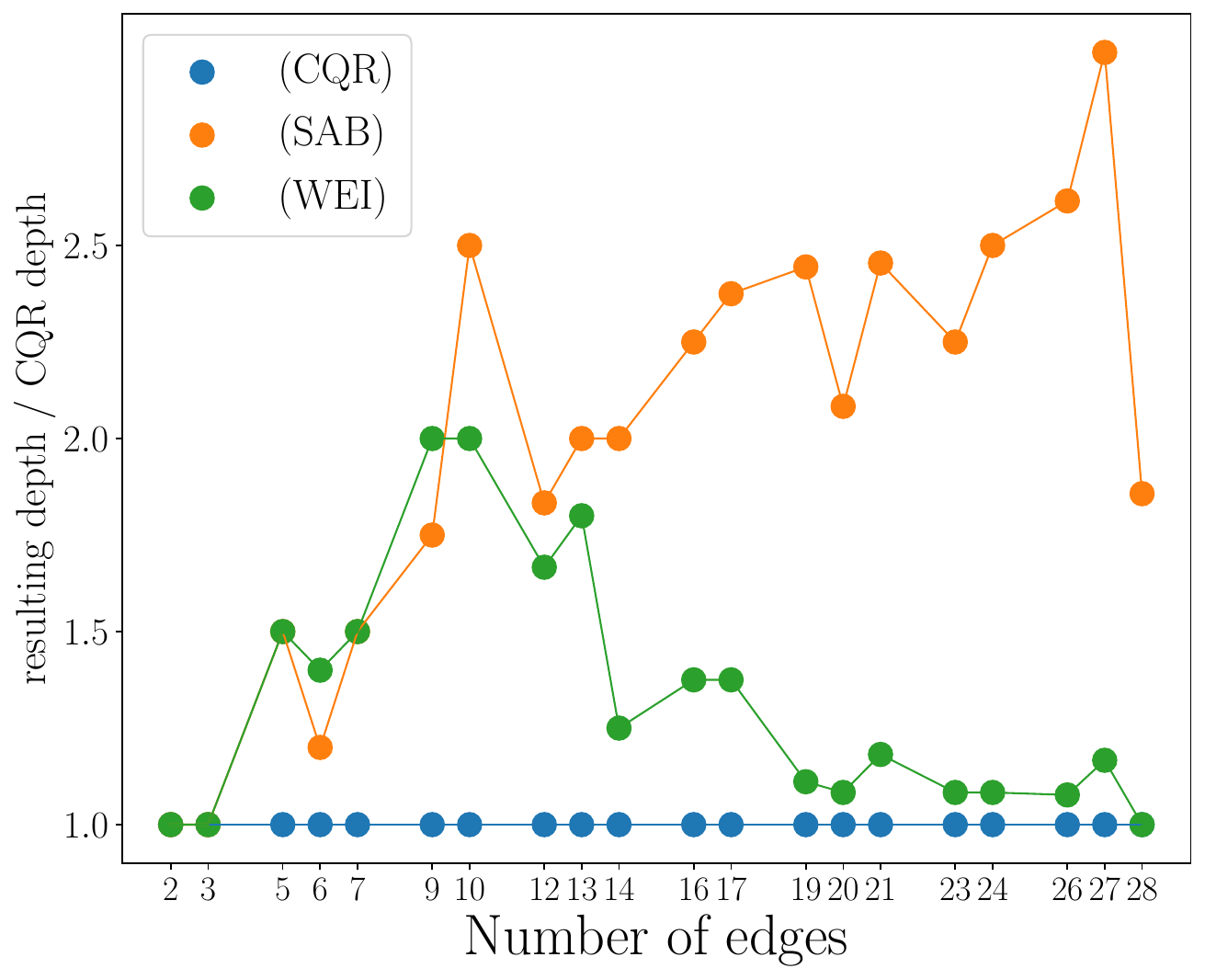}
        \caption{
            Depth comparison on $H_2$.
        }
        \label{fig:eval_h2_depth_scaled}
    \end{subfigure}
    \begin{subfigure}[h]{0.4\linewidth}
    \centering
    \includegraphics[width=\textwidth]{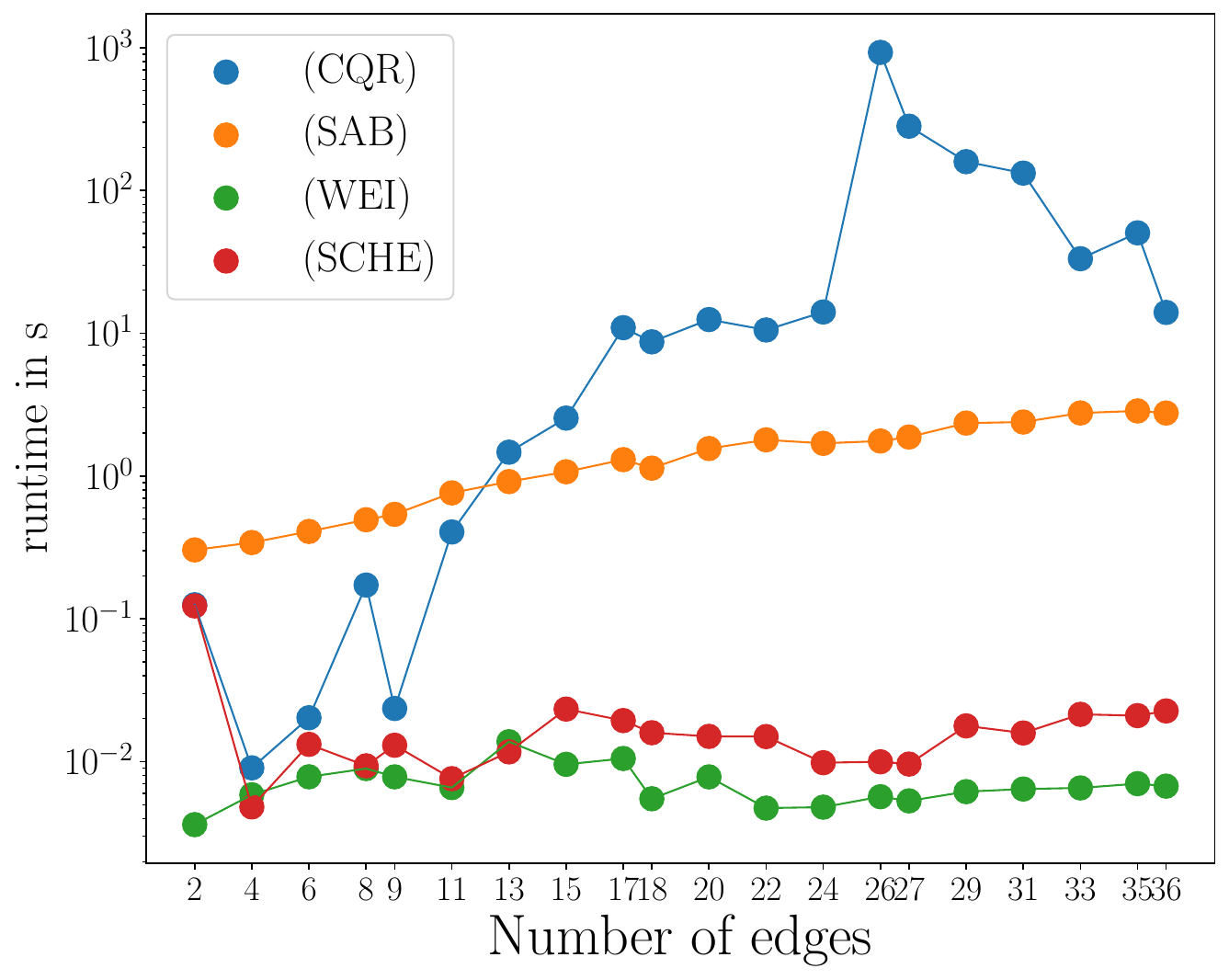}
    \caption{
        Runtime comparison on $H_1$.
    }
    \label{fig:eval_h1_time_logscale}
    \end{subfigure}
    \hfill
    \begin{subfigure}[h]{0.4\linewidth}
    \centering
    \includegraphics[width=\textwidth]{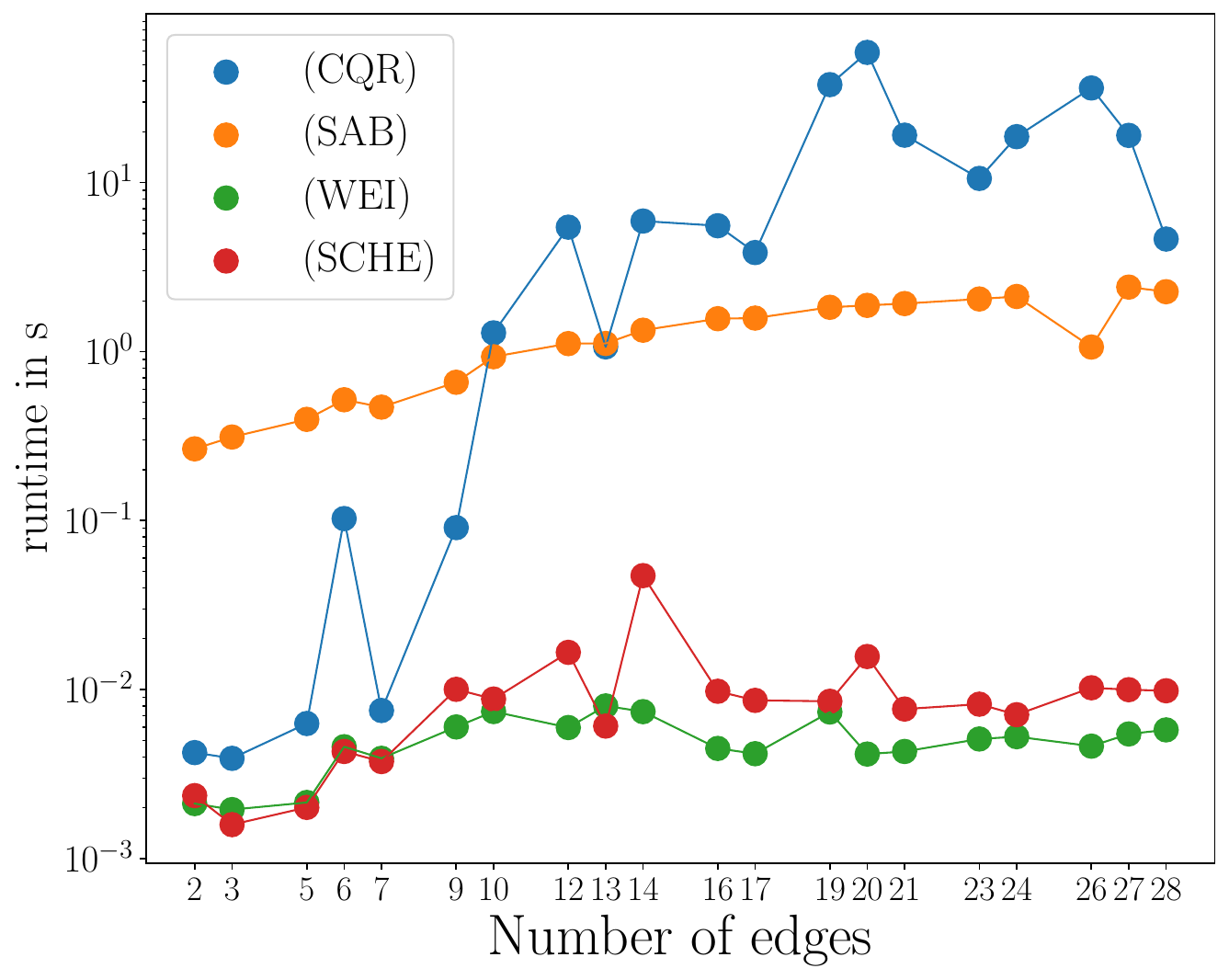}
    \caption{
        Runtime comparison on $H_2$.
    }
    \label{fig:eval_h2_time_logscale}
    \end{subfigure}
    \caption{Comparison with heuristics.}
    \label{fig:evalh1h2_heur}
\end{figure}
\begin{figure}
    \centering
    \subfloat{
        \includegraphics[width=0.425\textwidth]{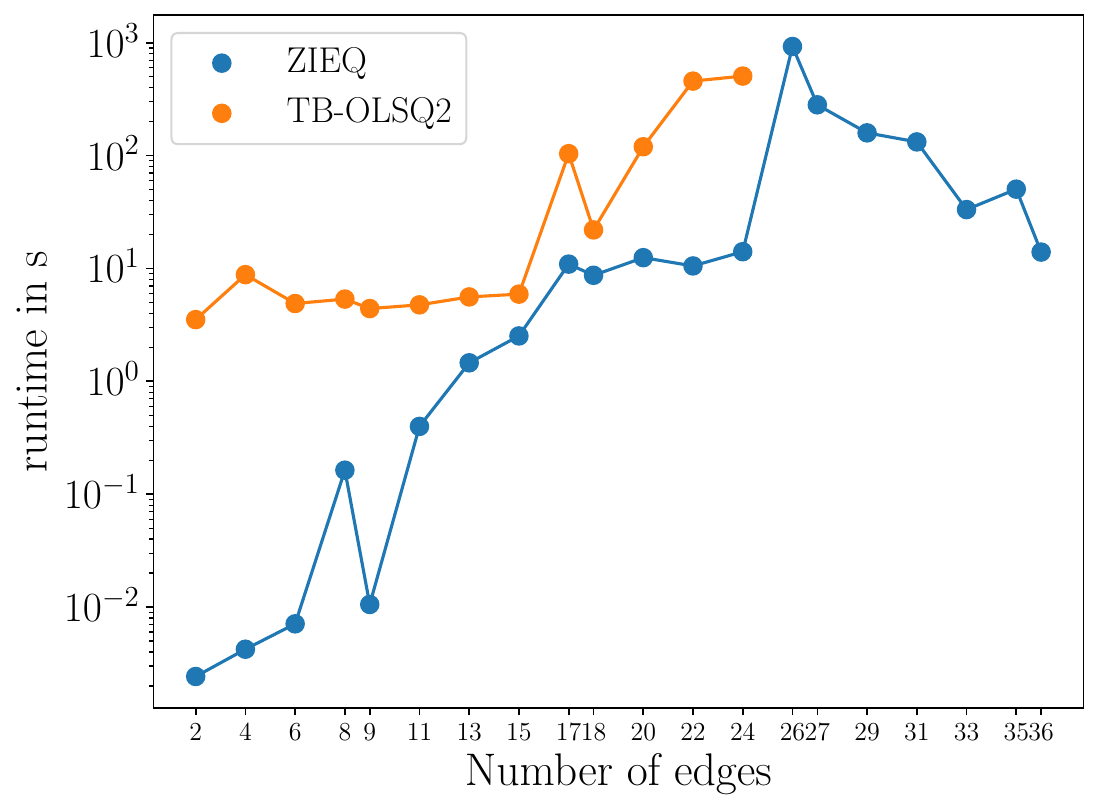}
    }
    \hspace{0.09\textwidth}
    \subfloat{
        \includegraphics[width=0.425\textwidth]{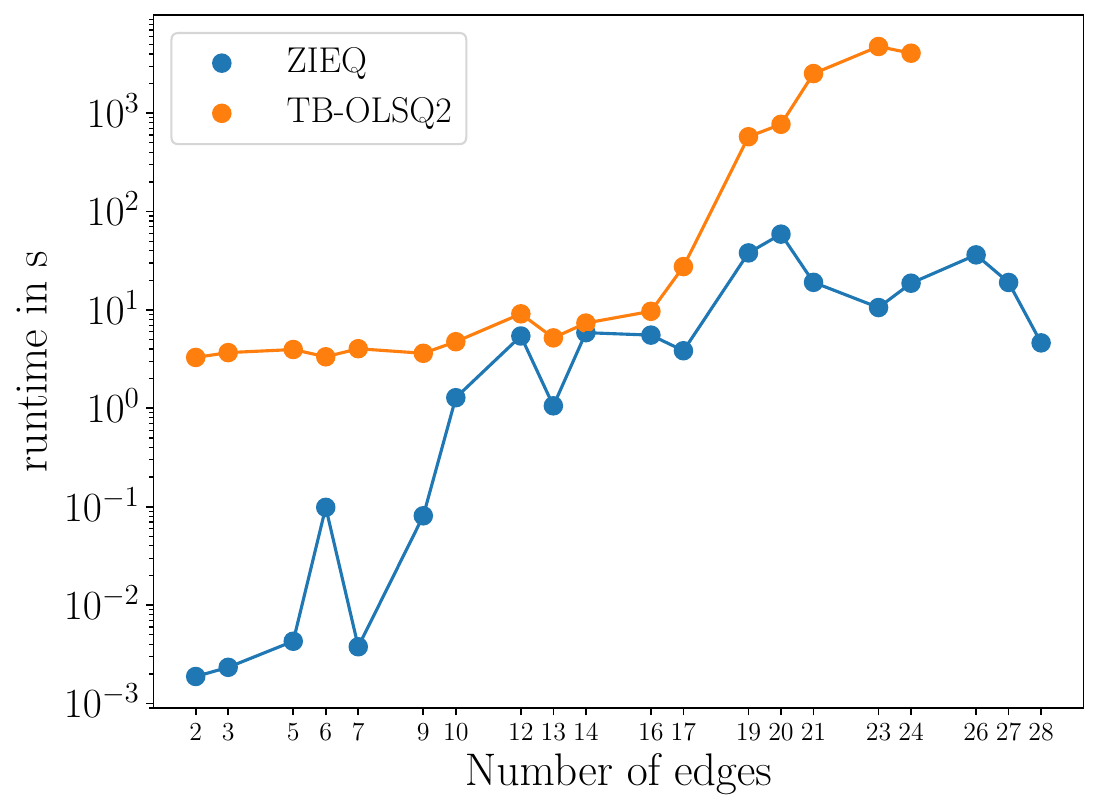}
    }
    \caption{
        Runtime comparison between \protect\hyperlink{model:BPZ<}{\text{ZIEQ}} and TB-OLSQ2 for $H_1$ (left) and $H_2$ (right).
    }
    \label{fig:eval_olsqtimesvscqrtimes}
\end{figure}

\textbf{Swap count.}
On average, (SAB) uses 49 \% and 56 \% more swaps than (CQR) for $H_1$ and $H_2$, respectively.
By construction, (WEI) uses the same number of swaps as (CQR) for the complete algorithm graph.
However, for instances with a sparse algorithm graph, (WEI) uses up to 10 times more swaps than (CQR), indicating that
the performance of (WEI) heavily depends on the instance density.
    
\textbf{Circuit depth.}
The depth gap between (CQR) and (SAB) tends to increase with the number of edges in the algorithm graph.
For dense algorithm graphs, (SAB) consistently returns circuits more than twice as deep as (CQR).
For (WEI), the depth gap to (CQR) decreases with the number of edges while (WEI) returns at most twice as deep circuits.
We conclude that for dense circuits, (WEI) yields routed circuit of low depth if given step-optimal solutions for $(H, K_n)$ and minimum edge colorings.

\textbf{Runtime.}
Being an exact approach, (CQR) 
requires significantly more time than (SAB) and (WEI), with maximum runtimes of 926 s and 59 s for $H_1$ and $H_2$.
The maximum runtime of (SAB) and (WEI) are consistently in the order of 2 s and 0.01 s, respectively.
For small-sized instances, (CQR) is faster than (SAB), which is due to the settings chosen for (SAB).
(SCHE) is solved very quickly, requiring at most 0.12 s for both $H_1$ and $H_2$.

In summary, (WEI) is effective when a high-quality swap strategy and a minimum edge coloring of the hardware graph are known and the algorithm graph is dense.
(SAB) is well-suited for sparse algorithm graphs.
(CQR) yields the best overall results in terms of both swap count and depth with significant improvements, using up to ten times less swaps and up to three times shallower circuits.
This improvement comes at the cost higher runtime.
However, for applications like QAOA where many different circuits are executed that all have the same algorithm graph, this time needs to be invested only once.
Moreover, the strong reduction in swap count and depth significantly reduce the error rates on current noisy devices.
Finally, we note that the scheduling approach (SCHE) is several orders of magnitude faster than the integrated algorithm and thus might be applicable to much larger instances, yielding high-quality schedules for approaches like (WEI).

\textbf{Comparison with a near-exact method.}
Next, we compare to TB-OLSQ2~\cite{lin2023scalable}, a near-optimal state-of-the-art SMT-based approach for solving the qubit routing problem with commuting gates.
We did not compare to \cite{yang2024quantum} or \cite{shaik2024optimal}, since \cite{shaik2024optimal} applies only to circuits with CNOT two-qubit gates, and \cite{yang2024quantum} did not publicly share their code.
Overall, we compare model \hyperlink{model:BPZ<}{\text{ZIEQ}} to TB-OLSQ2.
Specifically, we use the first step of their TB-OLSQ2-QAOA procedure, which applies the TB-OLSQ2 model to minimize the number of swaps.
Their model outputs a sequence of gate blocks and swap gates without explicit scheduling, which is equivalent to a TMP solution.
Therefore, we do not perform circuit scheduling but directly compare \hyperlink{model:BPZ<}{\text{ZIEQ}} with TB-OLSQ2 in terms of runtime and the number of swaps.
We use the \texttt{z3-solver} 4.15.1.0~\cite{de2008z3} as well as \texttt{pySAT} 1.8.dev17 for TB-OLSQ2, and impose a time limit of 7200 s for both TB-OLSQ2 and \hyperlink{model:BPZ<}{\text{ZIEQ}}.
Since both methods solve the TMP problem for a fixed number of steps and terminate upon finding a feasible solution, they should yield a solution with the same number of swaps.
Indeed, TB-OLSQ2 returned solutions with the same swap count for all instances but on the three instances of $H_1$ with 4, 6 and 9 nodes and the two instances of $H_2$ with 3 and 7 edges, where TB-OLSQ2 returned solutions with one more swap than \hyperlink{model:BPZ<}{\text{ZIEQ}}.
The runtimes are displayed in \Cref{fig:eval_olsqtimesvscqrtimes}.
TB-OLSQ2 failed to terminate within 7200 s for instances with 26+ edges for both $H_1$ and $H_2$.
For the instances it did solve, TB-OLSQ2 consistently required significantly more runtime.
Although TB-OLSQ2 solves small instances quickly, the runtime gap to \hyperlink{model:BPZ<}{\text{ZIEQ}} increases rapidly with the size of the instances.

\section{Conclusion}
In this work, we addressed the problem of qubit routing with commuting gates, which arises during the compilation of specific quantum algorithms, for example, the well-known Quantum Approximate Optimization Algorithm.
We developed a two-step decomposition approach, which is optimal in terms of inserted swap gates.
The first part consists of solving an instance of the Token Meeting Problem (TMP), whose solution yields a sequence of swap layers.
From a TMP solution, we construct a routed quantum circuit via a scheduling step.
The scheduling inserts gates into either swap layers or new layers, aiming to minimize the depth of the resulting circuit.

Via analysis of the TMP, we have shown that any instance of qubit routing with commuting gates can be solved using $\mathcal{O}(n^2)$ swaps.
On the other hand, there exist instances which require $\Omega(n^2)$ swaps. 
Having proven $\mathcal{NP}$-hardness, we proposed to use integer programming methods for solving TMP instances.
To this end, we developed and compared several integer programming models.
Among these, the \hyperlink{model:BPZ<}{\text{ZIEQ}} model showed the best computational performance.
We investigated symmetry-breaking techniques, which reduced runtime by a factor of up to $12$ in our experiments for instances with the complete algorithm graph.
Moreover, we derived linear descriptions of polytopes appearing as substructures in our proposed models, which lead to a speed-up of up to a factor of roughly 5 in our experiments.
Since we have proven these linear descriptions in a generalized setting, they might have applications beyond this work.
Furthermore, we proposed an integer program for the scheduling step which has runtimes below $0.12$~s in our experiments and might thus scale to larger circuit sizes.
Our computational experiments indicate that although exact models do not scale to large instances, they provide valuable insights in the performance of heuristics.

Our work leaves interesting open problems for further research:
\begin{itemize}
    \item Is \Cref{conjecture:tree_steps} true: can any TMP instance with $n$ nodes be solved in $\mathcal{O}(n)$ steps?
    \item Are there polynomial time approximation algorithms for STMP or PTMP for specific hardware graphs? Do the problems remain $\mathcal{NP}$-hard if one fixes the first bijection or considers specific hardware graphs?
    \item Could additional symmetry-breaking techniques or separating inequalities help speed up the solution of the integer programs?
\end{itemize}

\section*{Acknowledgements}
The authors thank Tobias Kuen, Christoph Hertrich, Eric Stopfer and Frauke Liers for helpful comments on the manuscript.
This research is part of the Munich Quantum Valley, which is supported by the Bavarian state government with funds from the Hightech Agenda Bayern Plus.

\bibliographystyle{unsrturl}  

\bibliography{./bibliography}

\appendix
\section{Proof of NP-completeness}\label{app:complexity}
\label{appendix}

In this section, we reduce \textsc{Subgraph Isomorphism}, which is well known to be $\mathcal{NP}$-complete~\cite{pruim2005complexity} to PTMP and STMP.

\begin{problem}\textsc{Subgraph Isomorphism}\\
	\textbf{Input:} A pair of undirected graphs $(G_1, G_2)$.\\
	\textbf{Question:} Is $G_2$ isomorphic to a subgraph of $G_1$?
\end{problem}

\begin{lemma}\label{lem:subgraph_isomorphism}
	Let $(H, A)$ be a TMP instance.
	Then,
	\begin{equation*}
		\ML(H, A)=0 \quad\iff \quad\MT(H, A)= 0 \quad\iff\quad  A \text{ is isomorphic to a subgraph of } H.
	\end{equation*}
\end{lemma}
\begin{proof}
	We obtain $\MT(H, A)=0 \iff \ML(H, A)=0$ from \Cref{lem:mt_ml_bounds}.
	On the other hand, we have $\MT(H, A)=0$ if and only if there is a solution $(f_1)$ such that for each connection $(p, q)$, we have $\{f_1(p),f_1(q)\}\in E$, which is equivalent to $A$ being isomorphic to a subgraph of $H$.
\end{proof}

\Cref{lem:subgraph_isomorphism} gives us a straightforward way to reduce \textsc{Subgraph Isomorphism} to both STMP and PTMP.
However, we have to modify the instance of \textsc{Subgraph Isomorphism} to two graphs of the same size such that the first graph is connected, which is required for both STMP and PTMP.

\begin{proof}[Proof of \Cref{thm:NPcompleteness}]
	We have already shown that STMP and PTMP are in $\mathcal{NP}$.
	For $\mathcal{NP}$-hardness, we reduce \textsc{Subgraph Isomorphism} to STMP and PTMP.
	Let $(G_1$, $G_2)$ be an instance of \textsc{Subgraph Isomorphism}.
	We construct the two connected graphs $G'_1, G'_2$ with $G'_i=(V(\hat{G}_i)\cup \{v_i^*\}, E(\hat{G}_i)\cup \{\{v_i^*, v\}: v\in V(\hat{G}_i)\})$, where $\hat{G}_i$ is the graph obtained by adding isolated nodes to $G_i$ until $|V(\hat{G_i})|=|V(G_1)|$ holds for $i=1,2$.
    See \Cref{fig:np_graph_transformation} for an illustration.
	Note that the encoding size of $G'_i$ is bounded by a polynomial in the encoding size of $G_1$ and $G_2$.
	By definition, $G_2$ is isomorphic to a subgraph of $G_1$ if and only if $\hat{G}_2$ is isomorphic to a subgraph of $\hat{G}_1$.
	We now show that $G_2$ is isomorphic to a subgraph of $G_1$ if and only if $G'_2$ is isomorphic to a subgraph of $G'_1$.
	The ``if'' direction holds by definition.
	For the ``only if'' direction, suppose that $G'_2$ is isomorphic to a subgraph of $G'_1$.
	Let $\sigma:V(G'_2)\mapsto V(G'_1)$ be a bijection such that $(i,j)\in E(G'_2)$ implies $(\sigma(i),\sigma(j))\in E(G'_1)$.
	We know that $\sigma$ maps $v_2^*$ to a node in $V(G'_1)$ which is connected to all other nodes.
	If $\sigma$ maps $v_2^*$ to a node $v^*\neq v_1^*$, we can ``exchange'' $v^*$ and $v_1^*$.
	Consequently, we define the bijection $\sigma':V(G'_2)\mapsto V(G'_1)$ with $\sigma'(v_2^*)=v_1^*$, $\sigma'(\sigma^{-1}(v_1^*))=\sigma(v_2^*)$, and $\sigma'(v)=\sigma(v)$ for all other nodes $v\in V(G'_2)$.
	Thus, $(i,j)\in E(G'_2)$ still implies $(\sigma'(i),\sigma'(j))\in E(G'_1)$.
	Restricting $\sigma'$ to the nodes of $\hat{G}_1$, we obtain that $\hat{G}_2$ is isomorphic to a subgraph of $\hat{G}_1$ and thus also $G_2$ is isomorphic to a subgraph of $G_1$.
	Finally, for both STMP and PTMP, we construct the instance with $A=G'_2$, $H=G'_1$, and $k=0$.
	\Cref{lem:subgraph_isomorphism} establishes that the answer to STMP and PTMP is yes if and only if the answer to \textsc{Subgraph Isomorphism} is yes.
\end{proof}

\begin{figure}
\centering
\includegraphics[page=12, width=0.675\textwidth]{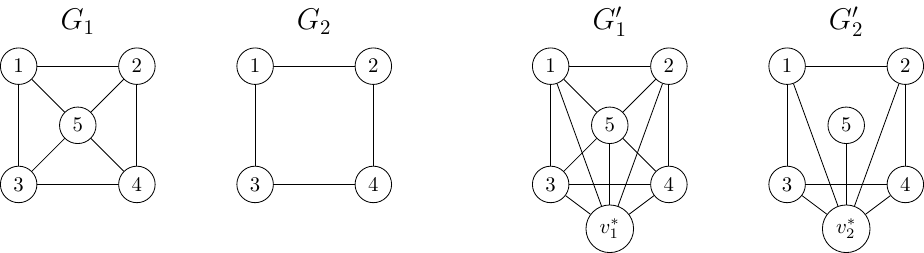}
\caption{Transforming a \textsc{Subgraph Isomorphism} instance $(G_1, G_2)$ into a STMP and PTMP instance $(G'_1, G'_2)$ as defined in the proof of \Cref{thm:NPcompleteness}.}
\label{fig:np_graph_transformation}
\end{figure}

\section{Linear Descriptions of Polytopes}\label{app:polytope_proofs}
Given a set of variables $x$ defined over a set $S$ and a subset $I\subseteq S$, we write $x(I)\define \sum_{i\in I}x_i$.
For an overview of polyhedral theory within integer programming, we refer to \cite{schrijver1998theory}.
\subsection{Linearization 1}\label{app:sec_Ylinearization}
The goal of this subsection is to derive a linear description of $P^\leq_{pq}$.
Recall that 
\begin{align*}
    P^\leq_{pq}= \conv\{(w^t_{pi}, w^t_{qj}, y^t_{pqij})\in \{0,1\}^{T(2|V|+|\bar{E}|)}:\;& \eqref{model:1model_assignnode2qubit}, \eqref{model:1amodel_gate_y}-\eqref{model:1amodel_mccormick2}
    \}.
\end{align*}
More generally, we study the following polytope.
\begin{definition}
    Let $G=(X\;\dot{\cup}\; Y, E)$ be a bipartite graph and let $\mathcal{P}_X$ and $\mathcal{P}_Y$ be partitions of $X$ and $Y$, respectively, that is, $\dot{\bigcup}_{I\in \mathcal{P}_X}I=X$ and $\dot{\bigcup}_{J\in \mathcal{P}_Y}J=Y$.
    We define
    \begin{align*}
        Q_{\mathcal{P}_{XY}}(G)\define \conv\{(x,y,z)\in \{0,1\}^{X\cup Y \cup E}:\eqref{con:q1} - \eqref{con:q5}\},
    \end{align*}
    where
    \begin{align}
        &x(I)=1,\quad && I\in \mathcal{P}_X,\label{con:q1}\\
        &y(J)=1,\quad && J\in \mathcal{P}_Y,\label{con:q2}\\
        &z_{ij} \leq x_i,\quad &&\{i,j\}\in E,\label{con:q3}\\
        &z_{ij} \leq y_j,\quad &&\{i,j\}\in E,\label{con:q4}\\
        &z(E)\geq 1.&&\label{con:q5}
    \end{align}
\end{definition}
The polytope $P^\leq_{pq}$ is isomorphic to $Q_{\mathcal{P}_{XY}}(G_{pq})$, where $G_{pq}=(\bigcup_{t=1}^T V^t_p\,\dot{\cup}\, \bigcup_{t=1}^T V^t_q, E_{pq})$ is a bipartite graph with
\begin{align*}
    &V^t_p \define \{(t, p, i): i\in V\},\quad 
    V^t_q \define \{(t, q, i): i\in V\},\\
    &E_{pq}\define \{\{(t, p, i), (t, q, j)\}: (i, j)\in \bar{E},\;t\in [T]\},
\end{align*}
and the partitions $\mathcal{P}_X=\{V^1_p,\dots,V^T_p\}$ and $\mathcal{P}_Y=\{V^1_q,\dots,V^T_q\}$.
We now provide a linear description for a special case of the polytope $Q_{\mathcal{P}_{XY}}(G).$
\begin{theorem}\label{thm:convexhull_PXY_ieq_expanded}
    Let $G=(X\,\dot{\cup}\, Y, E)$ be a bipartite graph and let $\mathcal{P}_X$ and $\mathcal{P}_Y$ be partitions of $X$ and $Y$,
    respectively, such that for all $i\in X$ there is a $J\in \mathcal{P}_Y$ with $N(i)\subseteq J$,
    and for all $j\in Y$ there is a $I\in \mathcal{P}_X$ with $N(j)\subseteq I$.
    Then,
    \begin{equation*}
        Q_{\mathcal{P}_{XY}}(G)=P\define  \{(x,y,z)\in [0,1]^{X\cup Y\cup E}: \eqref{con:q1},\eqref{con:q2},\eqref{con:q5}-\eqref{con:q7}\},
    \end{equation*}
    where
    \begin{align}
        &\sum_{j \in N(i)}z_{ij} \leq x_i,\quad && i \in X,\label{con:q6}\\
        &\sum_{i \in N(j)}z_{ij} \leq y_j,\quad && j \in Y.\label{con:q7}
    \end{align}
\end{theorem}
\begin{proof}
    First, we show that $Q_{\mathcal{P}_{XY}}(G)\subseteq P$.
    To this end, we prove that the constraints \eqref{con:q6} and \eqref{con:q7} are valid for $Q_{\mathcal{P}_{XY}}(G)$.
    For an $i\in X$ and a corresponding set $J\in \mathcal{P}_Y$ such that $N(i)\subseteq J$, we have
    \begin{equation*}
        x_i \overset{y(J)=1}{=} x_i \sum_{j\in J}y_j \geq x_i \sum_{j\in N(i)}y_j  \geq \sum_{j\in N(i)}z_{ij},
    \end{equation*}
    which proves that the inequality is valid for $Q_{\mathcal{P}_{XY}}(G)$.
    The validity of the constraints \eqref{con:q7} follows analogously.
    It is easy to see that all integral vertices of $P$ are also in $Q_{\mathcal{P}_{XY}}(G)$.
    The theorem now follows from the fact that all vertices of $P$ are integral.
    The vertices of $P$ are integral since the constraint matrix of $P$ is totally unimodular, which follows from noting that the constraint matrix of $P$ has the form of the totally unimodular matrix in \Cref{proof:partition_tu} up to changing the signs of some columns (which preserves total unimodularity).
\end{proof}
\begin{lemma}\label{proof:partition_tu}
    Let $G=\{X\,\dot{\cup}\, Y, E\}$ be a bipartite graph and let $(X_i)_{i\in [k]}$ and $(Y_j)_{j\in [l]}$ be families of disjoint subsets from $X$ and $Y$, that is, $\dot{\bigcup}_{i\in [k]}X_i\subseteq X$ and $\dot{\bigcup}_{j\in [l]}Y_j\subseteq Y$.
    Let $m: X\cup Y\cup E\to [|X|+|Y|+|E|]$ be a bijection that maps the set of edges $E$ to the last $|E|$ indices.
    Let $1_{S}\in \{0,1\}^{|X| +|Y| + |E|}$ be the vector with ones at indices $m(S)$.
    Lastly, let $I(G)$ be the incidence matrix of $G$ such that its $j$-th row corresponds to the node $m^{-1}(j)$.
    Then, the following matrix is totally unimodular
    \begin{equation*}
        M =
        \left(\begin{array}{cc}
            1_{X_1}^\top & 0 \\
            \vdots &  \vdots \\
            1_{X_k}^\top & 0 \\
            1_{Y_1}^\top & 0 \\
            \vdots & \vdots \\
            1_{Y_l}^\top & 0 \\ \hline \rule{0pt}{1\normalbaselineskip}
            \text{\rm Id}_{|X|+|Y|} & I(G)\\\hline \rule{0pt}{1\normalbaselineskip}
            0 & 1_E^\top
        \end{array}\right).
    \end{equation*}
\end{lemma}
\begin{proof}
    Let $L$ denote the last row-index of $M$, and let $M_i$ denote the $i$-th row vector of $M$.
    We will show that each row-submatrix $M_{I}$ of $M$ with row set $I\subseteq [L]$ admits an equitable row-bicoloring which is equivalent to $M$ being totally unimodular, see \cite{conforti2014integer}.
    An \emph{equitable row-bicoloring} $(R,B)$ of a matrix $A\in \mathbb{R}^{m\times n}$ with row vectors $a_i\in \mathbb{R}^n, i\in [m]$ is a partition of its rows into two sets $R\,\dot{\cup}\,B=[m]$ such that $\sum_{i\in R}a_i-\sum_{i\in B}a_i\in \{0,1\}^n$.
    We define
    \begin{align*}
        &U_x = [k], \quad U_y = [l]\setminus [k],
        \quad
        C_x = \{k+l+i: i\in m(X)\},
        \quad
        C_y = \{k+l+i: i\in m(Y)\}.
    \end{align*}
    $U_x$ and $U_y$ correspond to the $X$ and $Y$ subset rows in the upper block and $C_x, C_y$ to $X$ and $Y$ rows in the center block of $M$, respectively.
    
    \textbf{Case 1}: $L\notin I$.
    First, we construct an equitable row-bicoloring of $M_{[L-1]}$:
    Defining $R=U_x \cup C_y$ and $B=U_y \cup C_x$, one can verify that $(R, B)$ is an equitable row-bicoloring of $M_{[L-1]}$, since $(R,B)$ is simply an extension of the natural equitable row-bicoloring of the incidence matrix $I(G)$.
    Since $M_{[L-1]}$ has at most two nonzero entries in each column, the equitable row-bicoloring $(R,B)$ trivially induces one for any row-submatrix of $M_{[L-1]}$, including $M_I$.
    
    \textbf{Case 2}: $L\in I$.
    We show that $(R,B)$ with $R=(C_x\cup C_y)\cap I, B=(U_x\cup U_y\cup \{L\})\cap I $ is an equitable row-bicoloring of $M_{I}$.
    We first consider entries in the first $|X|+|Y|$ columns.
    Summing up the rows of $\text{Id}_{|X|+|Y|}$ with indices in $I$ gives a vector with entries in $\{0, 1\}$.
    Subtracting the rows with indices in $(U_x\cup U_y)\cap I$ leads to a reduction of at most one in each entry and we get a vector with entries in $\{0, \pm1\}$.
    Looking at the last $|E|$ columns, summing up the rows of $I(G)$ gives a vector with entries in $\{0, 1, 2\}$, subtracting the last row (with row-index $L$) yields a vector with $\{0, \pm 1\}$ entries.
\end{proof}
\Cref{convexhull_Ylinearization} follows directly from \Cref{thm:convexhull_PXY_ieq_expanded}.

\subsection{Linearization 2}\label{app:sec_Zlinearization}
The goal of this subsection is to derive linear description of the polytopes $P_{pqt}$ and $P_{pqt}^\leq$.
Recall that
\begin{align*}
    P_{pqt}&\define \conv\{(w^t_{pi},w^t_{qj},z^t_{pq})\in \{0,1\}^{2|V|+1}:\eqref{model:1model_assignnode2qubit},\; z^t_{pq}=\sum_{(i,j)\in \bar{E}}w^t_{pi}\cdot w^t_{qj}\},\\
    P^\leq_{pqt}&\define \conv\{(w^t_{pi},w^t_{qj},z^t_{pq})\in \{0,1\}^{2|V|+1}:\eqref{model:1model_assignnode2qubit},\; z^t_{pq}\leq \sum_{(i,j)\in \bar{E}}w^t_{pi}\cdot w^t_{qj}\}.
\end{align*}
More generally, we study the following polytopes.
\begin{definition}
    For a bipartite graph $G=(X\,\dot{\cup}\, Y, E)$, we define
\begin{align*}
    C_{XY}(G)&\define \conv\{(x,y,z)\in \{0,1\}^{|X|+|Y|+1}: z = \sum_{(i,j)\in E} x_i\cdot y_j ,\, x(X) = 1,\,y(Y) = 1\},\\
    C_{XY}^\leq(G)&\define \conv\{(x,y,z)\in \{0,1\}^{|X|+|Y|+1}: z \leq \sum_{(i,j)\in E} x_i\cdot y_j ,\, x(X) = 1,\,y(Y) = 1\}.
\end{align*}
\end{definition}
Suppose that we are given a TMP instance $(H,A)$ with $H=(V,E)$.
Let $V'$ be a copy of $V$ and define the bipartite graph $G_{pqt}\define(V\,\dot{\cup}\, V',\{(i,j)\in V\times V': \{i,j\}\in E\})$.
Then, $C_{XY}(G_{pqt})$ and $C_{XY}^\leq(G_{pqt})$ are isomorphic to $P_{pqt}$ and $P_{pqt}^\leq$, respectively.

If $G$ has no edges or is the complete bipartite graph, then $z=0$ and $z=1$ holds for $C_{XY}(G)$, respectively.
It is easy to see that this constraint, along with the constraints $x(X) = 1$ and $y(Y) = 1$, provides a linear description for $C_{XY}(G)$.
Now, the goal is to derive linear descriptions of $C_{XY}(G)$ and $C^\leq_{XY}(G)$ in the remaining cases.
For this, we use the fact that the polytope $C_{XY}(G)$ is related to the \emph{bipartite implication polytope}, introduced by Kuen et al.~\cite{ifthenpolytop}.
\begin{definition}[Bipartite Implication Polytope \cite{ifthenpolytop}]
    For positive integers $\alpha, \beta, \gamma\in \N_{\geq 1}$, and a surjective mapping $M:[\alpha] \times [\beta]\to [\gamma]$, define
    \begin{align*}
        BIP(M)\define\conv\{(x,y,z)\in \{0,1\}^{\alpha + \beta + \gamma}: \eqref{con:bip1} - \eqref{con:bip4}\},
    \end{align*}
    where
    \begin{align}
        &x_i\cdot y_j \leq z_{M(i, j)},\quad && (i, j)\in [\alpha] \times [\beta],\label{con:bip1}\\
        &x([\alpha]) = 1, &&\label{con:bip2}\\
        &y([\beta]) = 1, &&\label{con:bip3}\\
        &z([\gamma]) = 1. &&\label{con:bip4}
    \end{align}
\end{definition}
As stated in \cite{ifthenpolytop}, $
BIP(M)=\conv\left(\bigcup_{(i, j)\in [\alpha]\times [\beta]}\{e_i+e_{\alpha + j}+e_{\alpha + \beta + M(i, j)}\}\right)$ holds.
Based on this vertex description, we obtain the following equivalent description
\begin{align*}
    BIP(M)&=\conv\{(x,y,z)\in \{0,1\}^{\alpha + \beta + \gamma}: z_l = \sum_{(i,j)\in M^{-1}(l)} x_i\cdot y_j, l \in [\gamma],\, \eqref{con:bip2},\eqref{con:bip3}\}.
\end{align*}
Thus, $BIP(M)$ is a generalization of $C_{XY}(G)$.
Given a bipartite graph $G$, let $M_G$ be the mapping $M_G: X\times Y \to \{1,2\}$ that maps a pair $(i,j)$ to 1 if the pair forms an edge and to 2 otherwise.
Then, $(x,y,z,1-z)\in BIP(M_G)$ if and only if $ (x,y,z)\in C_{XY}(G)$ and thus the projection of $BIP(M_G)$ to the variables $(x,y,z_1)$ is isomorphic to $C_{XY}(G)$.

Kuen et al.~\cite{ifthenpolytop} prove that all facets belong to a certain class of inequalities.
Although it is not explicitly stated in their paper, the proof of their Lemma~4.1 extends also to implicit equalities.
\begin{lemma}[\cite{ifthenpolytop}]\label{ifthen:facets}
    Let $M:[\alpha]\times [\beta] \to [\gamma]$ be a surjective map.
    Then, any facet and implicit equality (other than $\eqref{con:bip2}-\eqref{con:bip4}$) is induced either by a lower bound or by an \emph{$n$-block inequality}
    \begin{equation}
    \label{eq:nblockieq}
    \sum_{m\in [n]}(x(I_m)+y(J_m)) \leq \sum_{l\in [\gamma]}c_l z_l + n,
    \end{equation}
    with subsets $I_n\subseteq \dots \subseteq I_1\subseteq [\alpha]$ and $J_1\subseteq \dots \subseteq J_n\subseteq [\beta]$ such that $\min_{l\in[\gamma]}c_l=0$ and
    \begin{equation*}
        c_l=\max_{(i,j)\in M^{-1}(l)}|\{k\in [n]:(i,j)\in I_k\times J_k\}|.
    \end{equation*}
\end{lemma}
Informally, $c_l$ is equal to the maximum number of blocks $I_k\times J_k$ that share the same index-pair $(i,j)$ for some $(i,j)\in M^{-1}(l)$.

Although \Cref{ifthen:facets} already provides a linear description for $C_{XY}(G)$, we now show that only a specific set of $n$-block inequalities can be facet-defining.
\begin{lemma}\label{ifthen:gamma2case}
    Let $M:[\alpha]\times [\beta] \to [2]$ be a surjective map.
    Then, any facet and implicit equality (other than $\eqref{con:bip2}-\eqref{con:bip4}$) of $BIP(M)$ is induced by a lower bound or by an $1$-block inequality such that its corresponding block $I\times J$ is contained in $M^{-1}(l)$ for some $l\in \{1,2\}$ and is inclusion-maximal with respect to this property, i.e., there is no block $I'\times J'$ contained in $M^{-1}(l)$ that strictly contains $I\times J$.
\end{lemma}
\begin{proof}
    Let $(I_1\times J_1,\dots, I_n\times J_n)$ be the blocks of an $n$-block inequality.
    By \Cref{ifthen:facets}, the $n$-block inequality can be rewritten to the form
    \begin{equation}\label{eq:nblockieqproof}
    \sum_{m\in [n]}(x(I_m)+y(J_m))\leq nz_l+n
    \end{equation}
    for some $l\in \{1,2\}$, which implies $I_k\times J_k\in M^{-1}(l)$ for all $k\in [n]$.
    W.l.o.g. let $l=1$.
    
    \textbf{Equations.}
    Suppose that \eqref{eq:nblockieqproof} is an implicit equality.
    It is easy to see that if the inequality is an implicit equality, then $a_i\in\{0,n\}$ and $b_j\in\{0,n\}$ must hold for all indices $i$ and $j$. 
    Thus, all blocks are identical and dividing all coefficients by $n$ leads to a 1-block constraint with the block $(I,J)$ for $I=I_1$ and $J=J_1$.
    We now show that $I\times J = M^{-1}(1)$, which proves that $I\times J$ is maximal with respect to containment in $M^{-1}(1)$.
    If there is an $(i,j)\in M^{-1}(1)\setminus I\times J$, then for $e_i+e_{\alpha+j}+e_{\alpha+\beta + 1}\in BIP(M)$ the inequality does not hold with equality, contradicting the assumption that the inequality is an implicit equality.
    
    \textbf{Facets.}
    Suppose that \eqref{eq:nblockieqproof} defines a facet.
    Since we have 
    \begin{equation*}
        \sum_{m\in [n]} (x(I_m)+y(J_m))- n - nz_1=\sum_{m\in [n]} (x(I_m)+y(J_m)- 1 - z_1)\leq 0,
    \end{equation*}
    the facet of the $n$-block inequality is contained in the intersection of faces induced by valid 1-block inequalities.
    Therefore, the facet itself must be induced by a 1-block inequality with the block $(I,J)$.
    Suppose that the block $(I,J)$ is not inclusion-maximal.
    Then, there is a block $(I',J')$ with $I'\times J'\subseteq M^{-1}(1)$ that strictly contains the block $(I,J)$.
    It is not hard to see that the 1-block inequality with the block $(I',J')$ then strictly dominates the facet-defining inequality, which contradicts the assumption that the inequality with block $(I,J)$ was facet-defining.
    The theorem now follows by combining the results for facets and equations of $BIP(M)$ with \Cref{ifthen:facets}.
\end{proof}
The proof of \Cref{ifthen:gamma2case} implies that there is an additional valid equation for $BIP(M)$ that is linearly independent from the equations $\eqref{con:bip2}-\eqref{con:bip4}$ if and only if $M^{-1}(1)$ is equal to $I\times J$ for subsets $I\subseteq [\alpha],\ J\subseteq [\beta]$.

Before applying \Cref{ifthen:gamma2case} to $C_{XY}(G)$, we first introduce some definitions.
A tuple $(I,J)\subseteq X\times Y$ is a \emph{biclique} of $G$ if $I\times J\subseteq E$ and an \emph{antibiclique} of $G$ if $I\times J\subseteq (X\times Y) \setminus E$.
Let $B(G)$ denote the set of inclusion-maximal bicliques of $G$, that is, for each $(I,J)\in B(G)$, there is no biclique $(I',J')$ such that $I\times J\subset I'\times J'$.
Analogously, let $AB(G)$ denote the set of inclusion-maximal antibicliques of $G$.
\begin{theorem}\label{thm:CXY=}
    For any bipartite graph $G$, we have
    \begin{equation*}
        C_{XY}(G)=\{(x,y,z)\in [0,1]^{|X|+|Y|+1}: x(X)=1,\,y(Y)=1,\,\eqref{con:ifthen_biclique_eq},\eqref{con:ifthen_biclique_ieq},\eqref{con:ifthen_antibiclique_ieq}\},
    \end{equation*}
    where
    \begin{align}
        &x(I)+y(J)= 1 + z,\quad && (I, J)\in B(G): E=I\times J,
        \label{con:ifthen_biclique_eq}\\
        &x(I)+y(J)\leq 1 + z,\quad && (I, J)\in B(G): E\neq I\times J,
        \label{con:ifthen_biclique_ieq}\\
        &x(I)+y(J)\leq 2 - z,\quad && (I, J)\in AB(G): E\neq X\setminus I\times Y\setminus J.
        \label{con:ifthen_antibiclique_ieq}
    \end{align}
\end{theorem}
\begin{proof}
    If $G$ has no edges, then $B(G)=\{(\emptyset, Y), (X,\emptyset)\}$ and \eqref{con:ifthen_biclique_eq} simplifies to $z=0$.
    If $G$ is the complete bipartite graph, then $B(G)=\{(X, Y)\}$ and \eqref{con:ifthen_biclique_eq} simplifies to $z=1$.
    For the remaining cases, the theorem follows by noting that $(x,y,z,1-z)\in BIP(M_G)\iff (x,y,z)\in C_{XY}(G)$, and by observing that the inclusion-maximal blocks of $M_G$ in \Cref{ifthen:gamma2case} correspond precisely to the inclusion-maximal bicliques and antibicliques in $G$.
\end{proof}
It is possible to prove a similar statement for $C_{XY}^\leq(G)$.
\begin{theorem}\label{thm:CXY<=}
    For any bipartite graph $G$, we have
    \begin{equation*}
        C_{XY}^\leq(G)=\{(x,y,z)\in [0,1]^{|X|+|Y|+1}: x(X)=1,\,y(Y)=1,\,\eqref{con:ifthen_biclique_eq},\eqref{con:ifthen_antibiclique_ieq}\}.
    \end{equation*}
\end{theorem}
\begin{proof}
    First, observe that $(x,y,z)\in C_{XY}(G)$ if and only if $(x,y,z,1-z)\in P$, where
    \[
    P\define\{(x,y,z)\in \{0,1\}^{|X|+|Y|+2}: z_1 \leq \sum_{(i,j)\in E} x_i\cdot y_j ,\, x(X) = 1,\,y(Y) = 1,\;z_1+z_2=1\}.
    \]
    With some minor easy adjustments, the proofs of Lemma 4.1. in \cite{ifthenpolytop} and \Cref{ifthen:gamma2case} can be carried out for $P$ as well.
    Thus, \Cref{ifthen:facets} and \Cref{ifthen:gamma2case} also hold for $P$.
    The theorem now follows since constraints~\eqref{con:ifthen_biclique_ieq} are not valid for $P$ and the only remaining valid equations and maximal 1-block inequalities are given by constraints~\eqref{con:ifthen_biclique_eq} and \eqref{con:ifthen_antibiclique_ieq}.
\end{proof}

\Cref{convexhull_Zlinearization} follows directly from \Cref{thm:CXY=} and \Cref{thm:CXY<=}.

\subsubsection{Constructing a model with few constraints.}
\label{sec:few_constraints}
We can use 1-block inequalities to construct a polytope $B(M)$ such that 
\[
B(M)\cap \mathbb{Z}^{\alpha+\beta+\gamma} = BIP(M)\cap \mathbb{Z}^{\alpha+\beta+\gamma}.
\]
To achieve this, it suffices to cover $M^{-1}(l)$ for each $l\in [\gamma]$ with a collection of blocks that are contained $M^{-1}(l)$.
Since each $(i,j)\in [\alpha]\times [\beta]$ is contained in at least one covering block, the corresponding 1-block inequality will ensure that $x_i=y_j=1$ implies $z_l=1$.
Thus, by including the 1-block inequalities corresponding to these covering blocks, we obtain the polytope $B(M)$.
Informally, the larger the blocks in the covering, the stronger the model.

We now construct such coverings for $V\times V'$ in the special case of $M_{G_{pqt}}:V\times V'\to \{1,2\}$, which corresponds to the polytope $P_{pqt}$.
Recall that $M_{G_{pqt}}(i,j)=1$ if and only if $\{i,j\}\in E$.
First, observe that for each $i\in V$, we can create the following two blocks
\[
\{i\}\times N(i)\subseteq M_{G_{pqt}}^{-1}(1)
\qquad \text{and} \qquad
\{i\}\times V'\setminus N(i)\subseteq M_{G_{pqt}}^{-1}(2).
\]
Taking the union of these blocks over all $i\in V$ yields a covering of $V\times V'$.
Often, these blocks can be expanded further.
Specifically, if there exists a node $j\in V$ such that $N(i)\subseteq N(j)$, we can expand the block $\{i\}\times N(i)$ to $\{i,j\}\times N(i)\subseteq M_{G_{pqt}}^{-1}(1)$.
More generally, we can expand this block to:
\[
\bigcup_{j\in V: N(i)\subseteq N(j)}\{j\} \times N(i)\subseteq M_{G_{pqt}}^{-1}(1).
\]
Similarly, the block $\{i\}\times V'\setminus N(i)$ can be expanded to
\[
\bigcup_{j\in V: V'\setminus N(i)\subseteq V'\setminus N(j)}\{j\} \times V'\setminus N(i)\subseteq M_{G_{pqt}}^{-1}(2).
\]
The resulting covering of $V\times V'$ using these expanded blocks leads to the constraints~\eqref{model:2model_z_1block1_1} and~\eqref{model:2model_z_1block0_1}.
In our computational experiments, we achieved significantly better performance when we also included the symmetric constraints~\eqref{model:2model_z_1block1_2} and \eqref{model:2model_z_1block0_2}, which correspond to the blocks
\[
N(i)\times \bigcup_{j\in V': N(i)\subseteq N(j)}\{j\}\subseteq M_{G_{pqt}}^{-1}(1)
\quad \text{and} \quad
V\setminus N(i)\times \bigcup_{j\in V': V\setminus N(i)\subseteq V\setminus N(j)}\{j\}\subseteq M_{G_{pqt}}^{-1}(2)
\]
for each node $i\in V$.

\end{document}